\documentclass[a4paper,10pt]{article}
\usepackage{lmodern}

\usepackage{wrapfig}
\usepackage{amsmath,amsfonts,amssymb,amsthm}
\usepackage{mathrsfs}
\usepackage{url}
\usepackage{xcolor}
\usepackage{graphicx} 
\usepackage{subcaption}

\usepackage[utf8]{inputenc}
\usepackage[english]{babel}

\usepackage{hyperref}
\usepackage[capitalize,noabbrev]{cleveref} 
\usepackage{enumerate} 

\newtheoremstyle{theorem}{5pt}{5pt}{\itshape}{}{\bfseries}{.}{.5em}{}

\numberwithin{equation}{section}

\theoremstyle{theorem}
\newtheorem{theorem}{Theorem}[section]
\newtheorem{lemma}[theorem]{Lemma}

\newtheorem{proposition}[theorem]{Proposition}

\newtheorem{mainassumption}{Assumption}

\newtheorem{main}{Theorem}
\newtheorem{main2}[main]{Theorem}
\newtheorem{examples}{Example}

\theoremstyle{definition}
\newtheorem{definition}[theorem]{Definition}

\newtheorem*{notation*}{Notation}

\theoremstyle{remark}
\newtheorem{remark}[theorem]{Remark}
\newtheorem*{remark*}{Remark}

\usepackage{titlesec}
\titleformat{\section}{\normalfont\bfseries}{\large\thesection.}{.5em}{\large}
\titlespacing*{\section}{0pt}{3.5ex plus 1ex minus .2ex}{2.3ex plus .2ex}

\allowdisplaybreaks[3]

\newcommand{\beq}{\begin {eqnarray}}
\newcommand \eeq {\end {eqnarray}}
\newcommand \ba {\begin {eqnarray*}}
\newcommand \ea {\end  {eqnarray*}}

\newcommand{\R}{\mathbb{R}}
\newcommand{\N}{\mathbb{N}}
\newcommand{\Z}{\mathbb{Z}}

\newcommand\p{\partial}

\newcommand{\supp}{\operatorname{supp}}

\begin{document}
\title{\bf\Large Gel'fand's inverse problem for the graph Laplacian}
\author{Emilia Bl{\aa}sten, Hiroshi Isozaki, Matti Lassas and Jinpeng Lu}
\date{}

\AtEndDocument{\bigskip{\footnotesize%
  \textsc{Emilia Bl{\aa}sten: Computational Engineering, School of Engineering Science, LUT University, Lahti campus, 15210 Lahti, Finland} \par  
  \textit{Email address}: \texttt{emilia.blasten@iki.fi} \par
  
  \addvspace{\smallskipamount}

  \textsc{Hiroshi Isozaki: Graduate School of Pure and Applied Sciences, Professor Emeritus, University of Tsukuba, Tsukuba, 305-8571, Japan} \par  
  \textit{Email address}: \texttt{isozakih@math.tsukuba.ac.jp} \par
  
  \addvspace{\smallskipamount}

  \textsc{Matti Lassas: Department of Mathematics and Statistics, University of Helsinki, FI-00014 Helsinki, Finland} \par  
  \textit{Email address}: \texttt{matti.lassas@helsinki.fi} \par
  
   \addvspace{\smallskipamount}
   
     \textsc{Jinpeng Lu: Department of Mathematics and Statistics, University of Helsinki, FI-00014 Helsinki, Finland} \par  
     \textit{Email address}: \texttt{jinpeng.lu@helsinki.fi} \par
  
     
  
%
}}

\maketitle

\begin{abstract}
We study the discrete Gel'fand's inverse boundary spectral problem of determining a finite weighted graph. Suppose that the set of vertices of the graph is a union of two disjoint sets: $X=B\cup G$, where $B$ is called the set of the boundary vertices 
and $G$ is called the set of the interior vertices. We consider the case where the vertices in the set $G$ and the edges connecting them are unknown.
Assume that we are given the set $B$ and the pairs $(\lambda_j,\phi_j|_B)$, where $\lambda_j$
are the eigenvalues of the graph Laplacian and $\phi_j|_B$ are the values of the corresponding eigenfunctions at the vertices in $B$.
We show that the graph structure, namely the unknown vertices in $G$ and the edges connecting them, along with the weights, can be uniquely determined from the given data, if every boundary vertex is connected to only one interior vertex and the graph satisfies the following property:
any subset $S\subseteq G$ of cardinality $|S|\geqslant 2$ contains two extreme points.
A point $x\in S$ is called an extreme point of $S$ if there exists a point $z\in B$ such that
$x$ is the unique nearest point in $S$ from $z$ with respect to the graph distance. 
This property is valid for several standard types of lattices and their perturbations.
\end{abstract}


\section{Introduction}

In this paper, we consider the discrete version of Gel'fand's inverse boundary spectral problem, defined for a finite weighted graph
and the graph Laplacian on it. We assume that we are given the Neumann eigenvalues of
 the graph Laplacian and the values of the corresponding Neumann eigenfunctions at a pre-designated subset of vertices, called the boundary vertices. 
 
Gel'fand's inverse boundary spectral problem was originally formulated in  \cite{Ge}
for partial differential equations. For partial differential operators, one considers
an  $n$-dimensional Riemannian manifold $(M,g)$ with boundary and the Neumann eigenvalue problem
\beq
\label{eq:problem}
& & - \Delta_g \Phi_j(x) = \omega_j\Phi_j(x), \quad \text{for }x\in M,\\
& &\ \ \p_\nu \Phi_j|_{\p M} = 0,\quad
\eeq
where $\Delta_g$ is the Laplace--Beltrami operator with respect
to the Riemannian metric $g$ on $M$, and $\Phi_j:M\to \R$ are the eigenfunctions corresponding to the 
eigenvalues $\omega_j\in \R$. In local coordinates $(x^i)_{i=1}^n$, the Laplacian has the  representation
\begin{equation}
	\Delta_g u= {\rm det}(g)^{-\frac12} \sum_{i,j=1}^n  \frac{\partial}{\partial x^i}\left({\rm det}(g)^{\frac12} g^{ij} \frac{\partial}{\partial x^j} u\right),
\end{equation}
where $g(x)=[g_{ij}(x)]_{i,j=1}^n$, ${\rm det}(g) = {\rm det}(g_{ij}(x))$ and $[g^{ij}]^n_{i,j=1} = g(x)^{-1}$. 

Gel'fand's inverse problem is to find the topology, differential structure and Riemannian metric of $(M,g)$ when one is given the boundary $\p M$ and the pairs $( \omega_j,\Phi_j|_{\p M}),$ $j=1,2,\dots$, where $\omega_j$ are the Neumann eigenvalues and $\Phi_j|_{\p M}$ are the Dirichlet boundary values of the corresponding eigenfunctions. Here, the  eigenfunctions $\Phi_j$
are assumed to form a complete orthonormal family in $L^2(M)$. 
We review earlier results on this problem and the related problems in Section \ref{subsec: earlier results}.

To formulate the discrete Gel'fand's inverse problem,
we consider 
 a finite weighted graph.
 We use the following terminology. When $X$ is the set of vertices of a finite graph, we can declare any subset
 $B\subseteq X$ to be the set of the boundary vertices, denoted by $B=\p G$, and call the set $G=X{{-}} B$ the set of the interior vertices of $X$. This terminology is motivated by inverse problems where one typically aims to reconstruct objects in a set $\Omega\subseteq \R^n$ 
using observations on the boundary $\p \Omega$. In our case, we 
 aim to reconstruct objects in a vertex set $G\subseteq X$ 
from observations on the boundary $\p G$.

 For $x,y\in X={{G}}\cup\partial {{G}}$, we denote $x\sim y$ if
 there is an edge in the edge set $E$ connecting $x$ to $y$, that is, $\{x,y\}\in E$. Every edge $\{x,y\}\in E$ has 
 a weight $g_{xy}=g_{yx}>0$ and every vertex $x\in  {{G}}$ has a measure $\mu_x>0$. 
For a function $u:{{G}}\cup\partial {{G}} \to \mathbb{R}$ defined on the whole vertex set,
the graph Laplacian $\Delta_G$ on ${{G}}$ is defined by
\begin{equation}\label{Laplaciandef}
  \big(\Delta_G u\big)(x) = \frac{1}{\mu_x}\sum_{\substack{y\sim x\\y\in G\cup \partial G}} g_{xy}\big( u(y) - u(x)\big), \quad x\in {{G}},
\end{equation}
and the Neumann boundary value $\partial_{\nu} u$ of $u$ is defined by
\begin{equation}\label{Neumanndef}
\big(\partial_{\nu} u\big)(z)=\frac{1}{\mu_z}\sum_{\substack{x\sim z\\ x\in G}} g_{xz} \big( u(x) - u(z)
  \big), \quad z\in \partial G.
\end{equation}  
We consider the Neumann eigenvalue problem
 \beq
\label{eq:problem disc}
& & - \Delta_G\phi_j (x)= \lambda_j\phi_j(x), \quad \text{for }x\in{{G}},  \\
& &\ \ \p_\nu \phi_j|_{\p {{G}}} = 0. 
\eeq

The discrete Gel'fand's inverse problem is to find the set of interior vertices ${{G}}$, the edge structure
of $({{G}} \cup \p {{G}},E)$ and the weights $g,\mu$,
when one is given the boundary $\p G$ and the pairs $( \lambda_j,\phi_j|_{\p G}),$ $j=1,2,\dots,N$, $N=|G|$, where $|G|$ is the number of elements in $G$.
Here, $\{\phi_j\}_{j=1}^N$ is a complete orthonormal family
of eigenfunctions and their Dirichlet boundary values, $\phi_j|_{\p G}=\big(\phi_j(z) \big)_{z\in \p G}$, are vectors in $\R^{|\p G|}$.

We mention that with suitable choices of $g,\mu$, our definition of the graph Laplacian \eqref{Laplaciandef} includes widely used Laplacians in graph theory, in particular, the combinatorial Laplacian when $g,\mu\equiv 1$, and the normalized Laplacian when $g\equiv 1,\, \mu_x=\textrm{deg}(x)$. The spectra of these two particular operators are mostly unrelated for general graphs and were usually studied separately.


\smallskip
Solving the discrete Gel'fand's inverse problem is  not possible without further assumptions due to the existence of isospectral graphs, see \cite{isospectralgraphs,FK,Tan}.
One of the main difficulties we encounter in solving the problem is that the graph Laplacian can have nonzero eigenfunctions which vanish identically on a part of the graph. This phenomenon, intuitively caused by the symmetry of the graph, can make one part of the graph invisible to the spectral data measured at another part. Therefore one needs to impose appropriate assumptions. On one hand, the assumptions have to break some symmetry of the graph to make the inverse problem solvable, and also designate sufficiently many boundary vertices to measure data on. On the other, the assumptions need to include a large class of interesting graphs besides trees, since trees are already well-understood. In this paper, we introduce the Two-Points Condition (\cref{assumption}), and prove that the inverse boundary spectral problem on finite graphs is solvable with this assumption. 
Our result can be applied to detect local perturbations and recover potential functions on periodic lattices (\cite{AIM16,AIM}), in particular, to probe graphene defects from the scattering matrix. We will address potential applications in another work.  

\smallskip
We start by defining the notations for undirected
simple graphs, where weights on vertices and edges are
considered. These weights are related to physical situations where graph models are applicable.

\subsection{Finite graphs}\label{subsection-finitegraph}

A graph is generally denoted by a pair $(X,E)$ with $X$ being the set
of vertices and $E$ being the set of edges between vertices. A graph
$(X,E)$ is finite if both $X$ and $E$ are finite. 
A graph is said to be simple if there is at most one edge between any pair
of vertices and no edge between the same vertex. For undirected simple graphs, 
edges are two-element subsets of $X$. We endow a general
graph with the following additional structures that affect wave
propagation on the graph.
\begin{definition}[Weighted graph with boundary]
We say that $\mathbb G = (G,\partial G, E,
  \mu, g)$ is a \emph{weighted graph with boundary} if the following conditions are satisfied.
  \begin{itemize}
  \item $G\cup\partial G$ is the set of vertices (points), $G\cap \partial G=\emptyset$; $E$ is the set of edges. Elements of $G$ are called interior
    vertices and elements of $\partial G$ are called boundary
    vertices. We require $(G \cup \partial G,E)$ to be an undirected simple
    graph.
  \item $\mu : G\cup\partial G \to \R_+$ is the weight function on the
    set of vertices.
  \item $g : E\to\R_+$ is the weight function on the set of edges.
  \end{itemize}
\end{definition}


We use the following terminology. 
A graph with boundary $\mathbb G$ is
\emph{finite} if $(G \cup\partial G, E)$ is finite. 
Vertices $x$ and $y$
are adjacent, denoted by $x\sim y$, if $\{x,y\}\in E$, i.e., there is an edge connecting $x$ to $y$. When $x\sim y$, we
denote by $g_{xy}$, or equivalently $g_{yx}$, the weight of the edge
connecting $x$ to $y$. We write $\mu_x$ short for $\mu(x)$.

The degree of a vertex $x$ of $\mathbb{G}$ is defined as the number of
vertices connected to $x$ by edges in $E$, denoted by $\deg_E(x)$ or $\deg_{\mathbb{G}}(x)$. The neighbourhood
$N(\partial G)$ of $\partial G$ is defined by 
$$N(\partial G)=\{x\in G:
x\sim z \textrm{ for some }z\in \partial G\}\cup\partial G.$$
When the weights are not relevant in a specific context, we make use of the
notation $(G,\partial G,E)$ for an unweighted graph with boundary.

\begin{definition}[Paths and metric]\label{distance}
  Let $x,y\in G\cup\partial G$. A \emph{path} of $(G\cup \partial
  G,E)$ from $x$ to $y$ is a sequence of vertices $(v_j)_{j=0}^J$
  satisfying $v_0=x$, $v_J=y$ and $v_j\sim v_{j+1}$ for $j=0, \ldots,
  J-1$. The length of the path is $J$. The \emph{distance}
  between $x$ and $y$, denoted by $d(x,y)$, is the minimal length
  among all paths from $x$ to $y$. In other words, the distance $d(x,y)$ 
  is the minimal number of edges in paths that connect $x$ to $y$.
  The distance
   is defined to be infinite if there is no
  path from $x$ to $y$. An undirected graph $(G\cup \partial
  G,E)$ can be considered as a discrete metric space equipped with the distance function $d$.
An undirected graph is \emph{connected} if there exists a path between any pair of vertices.
\end{definition}

A graph with boundary $\mathbb G$ is said to be connected if $(G \cup\partial G, E)$ is connected. 
We say that $\mathbb G$ is \emph{strongly connected}, if it is still connected after one removes all edges connecting boundary vertices to boundary vertices (see Definition \ref{reduceddef}).

\smallskip
We remark that in our setting, any pair of adjacent vertices has distance $1$, while different choices of distances appear in other settings. If the graph sits in a manifold, it is more natural to use the intrinsic distance of the underlying manifold. For this type of graphs, additionally with geometric choices of weights, the graph Laplacian \eqref{Laplaciandef} can be used to approximate the standard Laplacians on the manifold, as long as the graphs are sufficiently dense (\cite{BIK1,BIK2,BIKL,Lu1}).

\begin{definition}\label{extreme}
  Given a subset $S\subseteq G$, we say a
  point $x_0\in S$ is an \emph{extreme point of $S$ with respect to
    $\partial G$}, if there exists a point $z\in \partial G$ such that
  $x_0$ is the unique nearest point in $S$ from $z$, with respect to
  the distance $d$ on $(G\cup\partial G,E)$.
\end{definition}

\begin{mainassumption}\label{assumption}
  We impose the following assumptions on the finite graph $(G,\partial G,E)$.
  \begin{enumerate}[\quad(1)]
  \item \label{as1case1} For any subset $S\subseteq G$ with
    cardinality at least $2$, there exist at least two extreme points
    of $S$ with respect to $\partial G$. We refer to this condition as
    the Two-Points Condition.
  \item \label{as1case2} For any $z\in\partial G$ and any pair of distinct points $x,y\in G$, if $x\sim z,\,y\sim z$, then $x\sim y$.
    \end{enumerate}
\end{mainassumption}

Note that \cref{as1case2} of \cref{assumption} is void if
every boundary point is connected to only one interior point. Hence
any graph can be adjusted to satisfy \cref{as1case2} by attaching an
additional edge to every boundary point and declaring the added vertices as the new boundary points. We remark that \cref{as1case2} is essential for proper wavefront behaviour (see \cref{wavefront}).

One can view the Two-Points Condition (\cref{as1case1} of
\cref{assumption}) as a criterion of choosing appropriate boundary
points for solving the inverse boundary spectral problem. As
an intuitive example in the continuous setting, any compact subset of
a square in $\R^2$ has at least two extreme points unless it is a
single point set. In this case, two extreme points can be chosen by
taking a point achieving the maximal height and a point achieving the minimal
height with respect to one edge of the square. The boundary points
realizing the extreme point condition are the vertical projections of
those two chosen points to the proper edges (see
\cref{fig_square2points}). Several types of graphs satisfying the
Two-Points Condition are discussed in \cref{section-examples}.

%
%
%
%
%
\begin{figure}[h]
  \begin{center}
    \includegraphics[width=4cm]{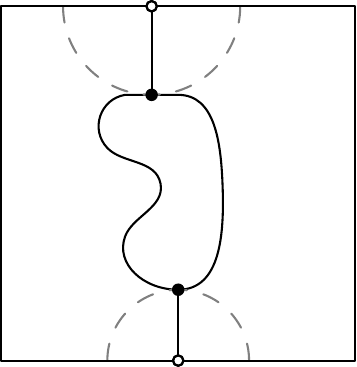}
    \caption{Any non-singleton compact subset of the unit square in
      $\R^2$ has at least two extreme points with respect to the
      boundary of the square.}
    \label{fig_square2points}
  \end{center}
  \vspace{-5mm}
\end{figure}

From now on, let $\mathbb{G}$ be a finite weighted graph with boundary. For a function $u:G\cup\partial G \to \mathbb{R}$, its graph Laplacian on $G$ is defined by the formula \eqref{Laplaciandef}.
Recall that the Neumann boundary value of $u$ is defined by the formula
\eqref{Neumanndef}, 
see e.g. \cite{C,ChungYau}.
For $u_1,u_2:G\cup\partial G\to \mathbb{R}$, we define the $L^2(G)$-inner product by
\begin{equation}\label{innerproduct}
\langle u_1 ,u_2 \rangle_{L^2(G)} = \sum_{x\in G} \mu_x u_1(x)u_2(x).
\end{equation}
For a finite graph with boundary, the function space $L^2(G)$ is exactly the space of real-valued functions on $G\cup\partial G$ equipped with the inner product (\ref{innerproduct}). Note that the inner product is calculated only on the interior $G$ and not on the boundary $\partial G$. The main reason for such consideration is that we mostly deal with functions $u$ satisfying $\partial_\nu u|_{\partial G} = 0$, in which case the values of $u$ on $\partial G$ are uniquely and linearly determined by the values on $G$, see \eqref{uValuesOnBndry}.

Let $q:G\to \mathbb{R}$ be a potential function, and we consider the following Neumann eigenvalue problem for the discrete Schr$\ddot{\textrm{o}}$dinger operator $-\Delta_G+q$.
\begin{equation} \label{eigenvalueproblem}
    \begin{cases}
      (-\Delta_G+q)u(x)=\lambda u(x), &x\in G, \lambda\in \mathbb{R},\\
      \partial_{\nu} u|_{\partial G}=0\, .
    \end{cases}
\end{equation}
Note that all Neumann eigenvalues are real, because the Neumann graph Laplacian is a self-adjoint operator on real-valued functions on $G$ with respect to the inner product (\ref{innerproduct}) due to Lemma \ref{Green}. 
In particular, the number of Neumann eigenvalues is equal to $|G|$, the number of interior vertices.

\begin{definition}\label{Neumannboundarydata}
  Let $\mathbb G$ be a finite weighted graph with boundary, and $q:G\to \mathbb{R}$
  be a potential function. A collection of data $(\lambda_j, \phi_j|_{\partial
    G})_{j=1}^N$ is called the \emph{Neumann boundary spectral data} of $(\mathbb{G},q)$, if
    
   $\bullet\;$ $\lambda_j\in\R$, $\phi_j : G\cup\partial G \to \R$, $N=|G|$
    is the number of interior vertices of $\mathbb G$;

   $\bullet\;$ the functions $\phi_j$ are Neumann eigenfunctions with respect to Neumann eigenvalues $\lambda_j$ for the equation (\ref{eigenvalueproblem}), namely
    \begin{equation} \label{neumannEigFunc}
      (-\Delta_G + q)\phi_j = \lambda_j \phi_j, \qquad \partial_\nu
      \phi_j|_{\partial G} = 0;
    \end{equation}
    
  $\bullet\;$ the functions $\phi_j$ form an orthonormal basis of $L^2(G)$.
\end{definition}

\begin{remark}

There are multiple choices of Neumann boundary spectral data for a given graph. More precisely,
  given two choices of Neumann boundary spectral data $(\lambda_j,\phi_j |_{\partial G})_{j=1}^N$ and $ (\tilde{\lambda}_j,\tilde{\phi}_j |_{\partial G})_{j=1}^N$ of
  $(\mathbb G, q)$, they are equivalent
  if
  
  (i) there exists a permutation $\sigma$ of $\{1,\ldots,N\}$ such that
    $\tilde{\lambda}_{\sigma(j)} = \lambda_j$ for all $j$;
    
  (ii) for any fixed $k$, there exists an orthogonal
    matrix $O$ such that
    \begin{equation*} \label{orthogonalTransf}
      \tilde{\phi}_{\sigma(i)}|_{\partial G} = \sum_{j\in L_k} O_{ij} \phi_j|_{\partial G},
    \end{equation*}
    for all $i\in L_k$, where $L_{k} = \{j |\, \lambda_j = \lambda_{k}\}$ and the matrix $O$ is of dimension $|L_k|$.
    
 In fact, this is the only non-uniqueness in the choice of Neumann boundary spectral data (a linear algebra fact).
In other words, there is exactly one equivalence
  class of Neumann boundary spectral data on any given finite weighted graph with
  boundary, and any representative of that class
  is a choice of Neumann boundary spectral data.
  We mention that the Neumann boundary spectral data is related to other types of data on graphs, such as the Neumann-to-Dirichlet map (see \cite{KKL,KKLM} for the manifold case).
\end{remark}

\smallskip
Next, we define our \emph{a priori} data. In order to uniquely determine the graph structure,
not only do we need to
know the Neumann boundary spectral data, but some structures related to the boundary also need to be known. In essence, this extra knowledge is the number of interior points connected to the
boundary, and the edge structure between the boundary and its
neighbourhood.
\begin{definition} \label{aPrioriGeom}
  Let $\mathbb G, \mathbb G'$ be two finite graphs with boundary. 
  We say that $\mathbb G, \mathbb G'$ are
  \emph{boundary-isomorphic}, if there exists a bijection $\Phi_0 :
  N(\partial G) \to N(\partial G')$ with the following properties: 
  \begin{itemize}
  \item [(i)] $\Phi_0|_{\partial G}:\partial G\to\partial G'$ is bijective;  
  \item [(ii)] for any $z\in \partial G$, $y\in N(\partial G)$, we have $y\sim z$
    if and only if $\Phi_0(y) \sim' \Phi_0(z)$, where $\sim'$ denotes the edge relation of $\mathbb{G}'$.
  \end{itemize}  
We call $\Phi_0$ a \emph{boundary-isomorphism}.
\end{definition}

\begin{definition} \label{aPrioriSpec}
  Let $\mathbb G, \mathbb G'$ be two finite weighted graphs with boundary, and $q,q'$ be real-valued potential functions on $G, G'$.
  We say $(\mathbb{G},q)$ is \emph{spectrally
  isomorphic} to $(\mathbb{G}',q')$ (with a boundary-isomorphism $\Phi_0$),  
  if
  \begin{itemize}
  \item [(i)] there exists a boundary-isomorphism $\Phi_0:N(\partial G) \to
    N(\partial G')$;
  \item  [(ii)] the Neumann boundary spectral data of $(\mathbb{G},q)$ and $(\mathbb{G}',q')$ have the same number of eigenvalues counting multiplicities;
  \item   [(iii)]there exists a choice of Neumann boundary spectral data of $(\mathbb{G},q)$ and $(\mathbb{G}',q')$, such that $\lambda_j=\lambda_j'$
    and $\phi_j|_{\partial G} = \phi_j' \circ \Phi_0|_{\partial G}$
    for all $j$.
  \end{itemize}
\end{definition}

Note that if $(\mathbb{G},q)$ and $(\mathbb{G}',q')$ are assumed to be spectrally isomorphic, in particular to have the same number of Neumann eigenvalues, then $\mathbb{G}$ and $\mathbb{G}'$ necessarily have the same number of interior vertices. Moreover, the existence of a boundary-isomorphism from the definition (i) implies that the number of boundary vertices is also necessarily the same.

\smallskip
Now we state our main results, Theorems \ref{structureThm} and
\ref{coefficientThm}.

\begin{main} \label{structureThm}
  Let $\mathbb G=(G,\partial G, E,
  \mu, g), \mathbb G'=(G',\partial G', E',
  \mu', g')$ be two finite, strongly connected, weighted graphs with boundary satisfying
  \cref{assumption}. Let $q,q'$ be real-valued potential functions on $G, G'$. Suppose $(\mathbb{G},q)$ is spectrally
  isomorphic to $(\mathbb{G}',q')$ with a boundary-isomorphism $\Phi_0$.
  Then there exists a bijection $\Phi:G\cup \partial G \to G'\cup\partial G'$ such that
  \begin{enumerate}[(1)]
  \item $\Phi|_{\partial G} = \Phi_0|_{\partial G}\,$, 
  \item for any pair of vertices $x_1,x_2$ of $\mathbb G$, we have
    $x_1\sim x_2$ if and only if $\Phi(x_1) \sim' \Phi(x_2)$.
  \end{enumerate}  
\end{main}

\begin{remark*}
  It may happen that $\Phi$ and $\Phi_0$ differ on $G\cap N(\partial G)$, for example if
  there exist points $y_1,y_2\in G\cap
  N(\partial G)$ which are connected to the same set of
  boundary vertices but connected to different parts in the interior.
\end{remark*}

\begin{main2} \label{coefficientThm}
  Take the assumptions of \cref{structureThm}, and identify vertices of $\mathbb{G}$ with vertices of $\mathbb G'$ via the bijection $\Phi$.
  Assume furthermore that $\mu_z=\mu'_z$,
  $g_{xz}=g'_{xz}$ for all $z\in\partial G$, $x\in G$, where $\mu',g'$ denote the weights of $\mathbb G'$. Then the following
  two conclusions hold.
  \begin{enumerate}[(1)]
  \item \label{coefCase1} If $\mu=\mu'$, then $g=g'$ and $q=q'$.
  
  \item \label{coefCase2} If $q=q'=0$, then $\mu=\mu'$ and $g=g'$.
  \end{enumerate}
  In particular, if $\mu={\rm deg}_{\mathbb{G}}$ and $\mu'={\rm deg}_{\mathbb{G}'}$, then $g=g'$ and $q=q'$.
\end{main2}

\subsection{Earlier results and related inverse problems}\label{subsec: earlier results}

Gel'fand's inverse  problem   \cite{Ge} for partial differential equations has been a paradigm problem in the study of the mathematical inverse problems and imaging problems arising from applied sciences. 
The combination of the boundary control method, pioneered by Belishev on domains of $\R^n$
and by Belishev and Kurylev on manifolds \cite{BelKur}, and the Tataru's unique continuation theorem \cite{Tataru1}  gave a solution to the inverse problem of determining the isometry type of a Riemannian manifold from given 
boundary spectral data. Generalizations and alternative methods to solve this problem have been studied e.g.\ in \cite{AKKLT,Bel-heat,BelKa,Caday,HLOS,KrKL,KOP,LassasOksanen}, see  additional references in \cite{Bel-review,KKL,L}.
 The inverse problems for the heat, wave and Schr\"odinger equations can be reduced 
 to Gel'fand's inverse  problem, see \cite{Bel-heat,KKL}. In fact, all these problems are 
 equivalent, see  \cite{KKLM}. Also, for the inverse problem for the wave equation with the measurement data  on a sufficiently large finite time interval, it is possible to continue the data to an infinite time interval, which makes 
 it possible to reduce the inverse problem to Gel'fand's inverse  problem, see \cite{KKL,KurLas}. %
 The stability of the solutions of these inverse problems have been analyzed in \cite{AKKLT,BKL3,BILL,FIKLN,StU}. 
Numerical methods to solve Gel'fand's inverse  problems
have been studied in \cite{Bel-num,HoopOksanen1,HoopOksanen2}.
The inverse boundary spectral problems have been extensively studied also for  elliptic equations on bounded  domains of $\R^n$. In this setting, Gel'fand's problems
can be solved by reducing it, see \cite{NSU,Novikov}, to 
  Calder\'on's inverse  problems for elliptic
equations that were solved using complex geometrical optics, see \cite{SyUl}.


An intermediate model between discrete and continuous models is the quantum graphs, namely graphs equipped with differential operators defined on the edges. In this model, a graph is viewed as glued intervals, and the spectral data that are measured are usually the spectra of differential operators on edges subject to the Kirchhoff condition at vertices. For
such graphs, two problems have attracted much attention. In the case where one uses only the spectra of differential operators as data,  Yurko (\cite{Y05,Y09,Y10}) and other researchers (\cite{AK,BrownW,K08}) have developed so called spectral methods to solve inverse problems.
Due to the existence of isospectral trees, one spectrum is not enough to determine the operator and therefore multiple measurements are necessary. It is known in \cite{Y10} that the potential can be recovered from appropriate spectral measurements of the Sturm-Liouville operator on any finite graph.
An alternative setting is to  consider inverse problems for quantum graphs when one is given the eigenvalues of the differential operator and the values of the eigenfunctions at some of the nodes. 
Avdonin and Belishev and other researchers (\cite{AK,Avdonin2,Avdonin3,Avdonin4,B,BV}) have shown that it is possible to solve a type of inverse spectral problem for trees (graphs without cycles). 
With this method, one can recover both the tree structures and differential operators. 
 
In this paper, we consider inverse problems in the purely discrete setting, that is, for the discrete graph Laplacian.
In this model, a graph is a discrete metric space with no differential structure on edges. The graph can be additionally assigned with weights on vertices and edges. The spectrum of the graph Laplacian on discrete graphs is an object of major interest in discrete mathematics (\cite{C,isospectralgraphs,S08}). It is well-known that the spectrum is closely related to geometric properties of graphs, such as the diameter (e.g. \cite{CFM94,CGY,ChungYau}) and the Cheeger constant (e.g. \cite{Cheeger,C05,F96}). There were inverse problems, especially the inverse scattering problem, considered on periodic graphs (e.g. \cite{A1,IK,KS}). However, due to the existence of isospectral graphs (\cite{isospectralgraphs,FK,Tan}), few results are known regarding the determination of the exact structure of a discrete graph from spectral data.

There have been several studies with the goal of determining the structure or weights of a discrete weighted graph from indirect measurements in the field of inverse problems. These studies mainly focused on the \emph{electrical impedance tomography on resistor networks} (\cite{Bor,CM00,LST}),
where electrical measurements are performed at a subset of vertices called the boundary. However, there are graph transformations which do not change the electrical data measured at the boundary, such as changing a triangle into a Y-junction, which makes it impossible to determine the exact structure of the inaccessible part of the network in this way. Instead, the focus was to determine the resistor values of given networks, or to find equivalence classes of networks (with unknown topology) that produce a given set of boundary data (\cite{CdV94,CdVGV96,CIM98}).

\begin{figure}
\centerline{
    \includegraphics[height=4cm]{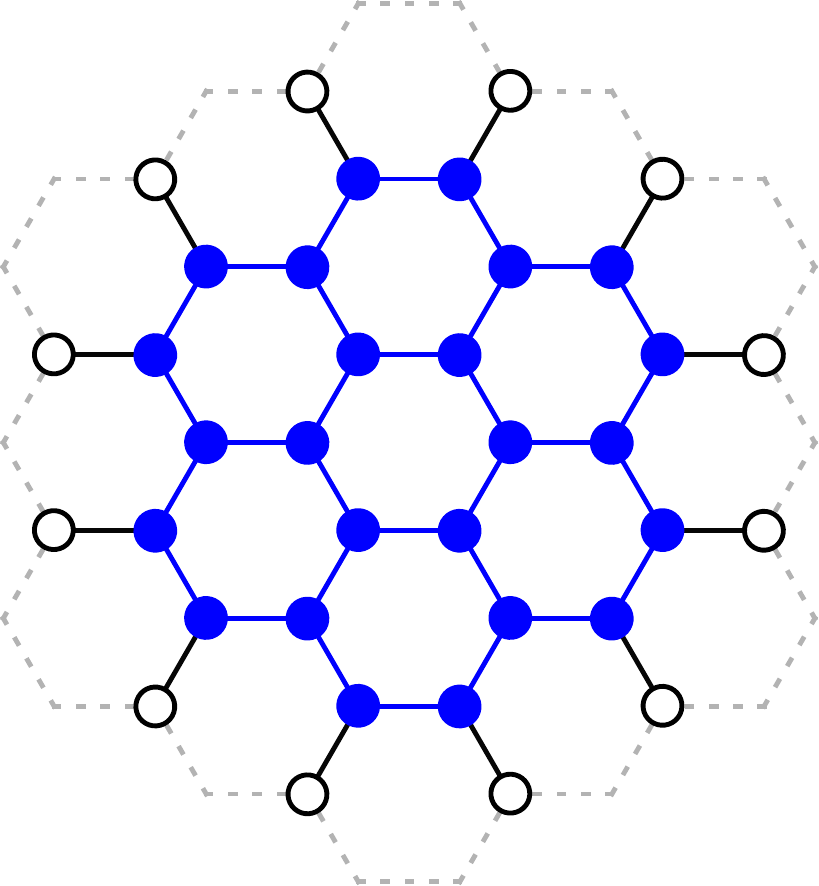}\hspace{3cm} 
        \includegraphics[height=4cm]{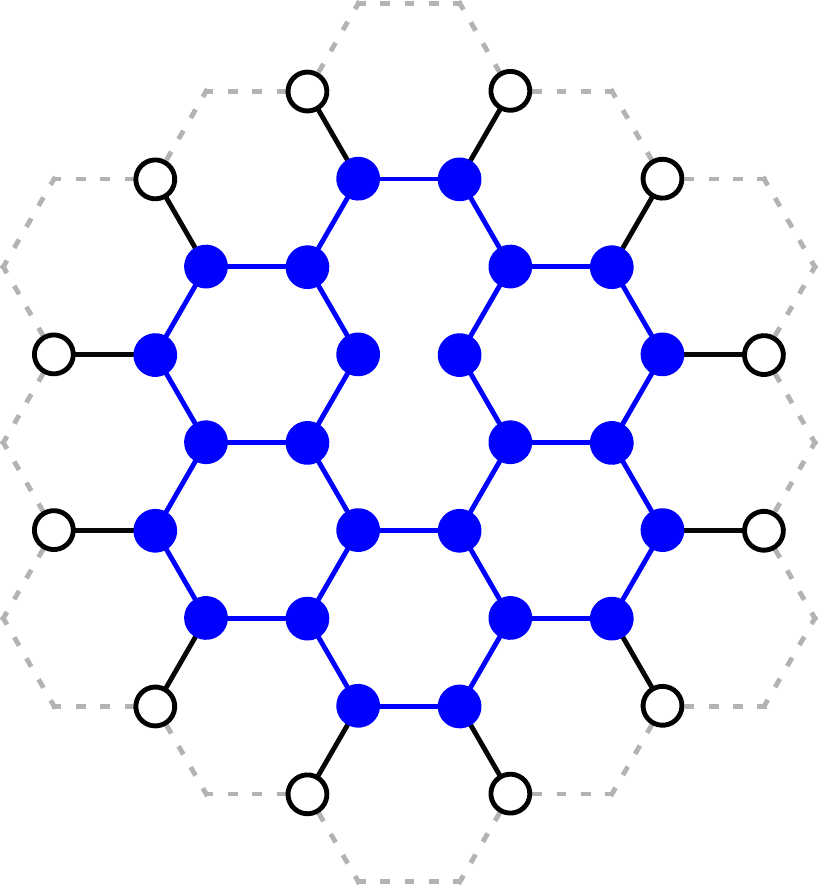}}
    \caption{{\bf Left:} Finite hexagonal lattice. The white vertices are considered to be the boundary vertices for the set of the blue (interior) vertices. {\bf Right:} A finite hexagonal lattice with one blue edge removed. 
Theorems \ref{structureThm} and
\ref{coefficientThm} show that the exact structure of such graphs can be uniquely recovered from the boundary spectral data.\vspace{-5mm}}
    \label{Finite hexagonal lattice}
\end{figure}

\begin{figure}
\centerline{
    \includegraphics[height=4cm]{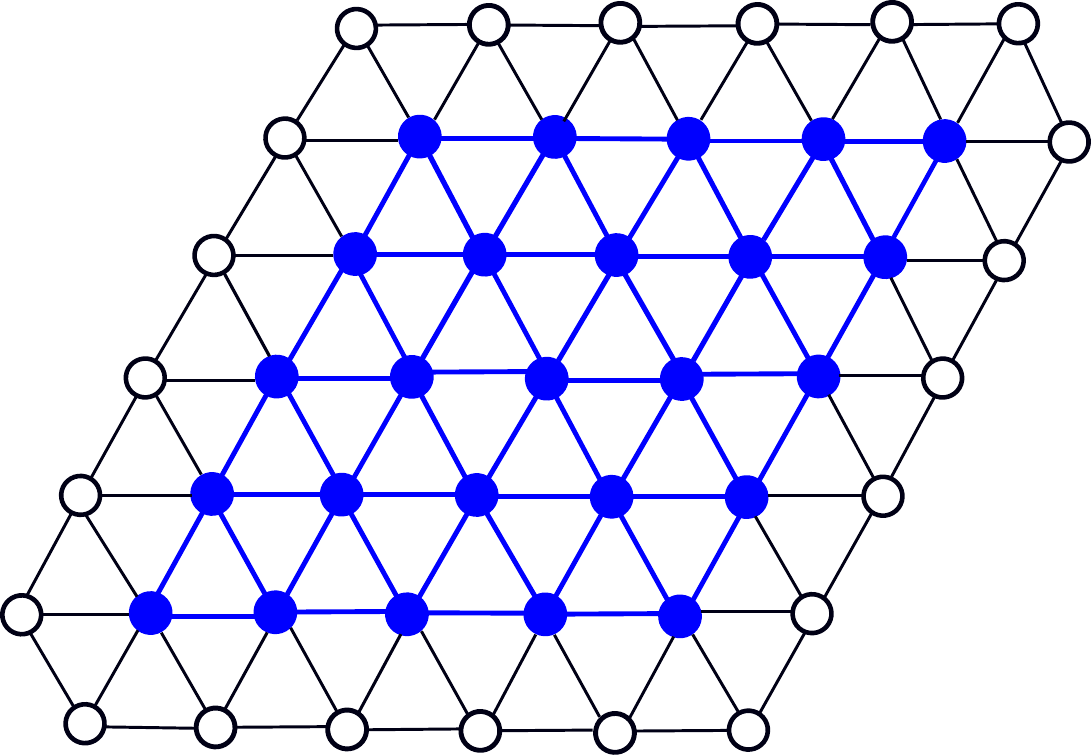}\hspace{1cm}   \includegraphics[height=4cm]{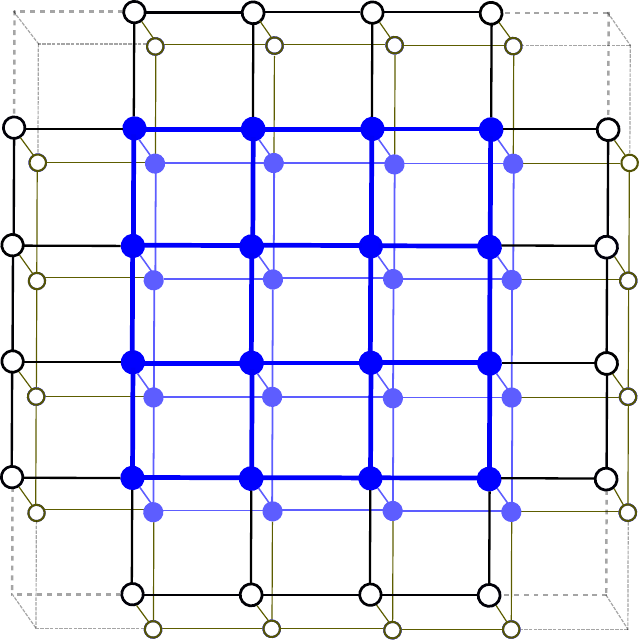}}
    \caption{{\bf Left:} Finite triangular lattice. The white vertices are the boundary vertices; the blue vertices are the interior vertices. {\bf Right:} Finite two-level square ladder, made out of two layers of square lattices.
    \vspace{-5mm}}
    \label{fig_ladder}
\end{figure}

\subsection{Examples}\label{section-examples}

As primary examples, we consider several standard types of graphs satisfying the Two-Points Condition (\cref{as1case1} of \cref{assumption}).

\begin{examples}
{\rm All finite trees satisfy the Two-Points Condition, with the boundary vertices being all vertices of degree $1$.}

{\rm This fact can be shown as follows. Recall that a tree is a connected graph containing no cycles. Take any subset $S$ of the interior vertices (i.e., vertices of degree at least $2$) of a tree with $|S|\geqslant 2$. If $|S|=2$, say $S=\{v_1,v_2\}$, then both vertices in $S$ are extreme points. This can be argued as follows. Take the (unique) shortest path between $v_1,v_2$, remove the edge incident to $v_1$ on this path, and one gets two disjoint subtrees. Consider the subtree containing $v_1$. One can see that any boundary vertex on this subtree realizes the extreme point condition for $v_1$. If $|S|>2$, pick two arbitrary points in $S$, say again $\{v_1,v_2\}$, then in the same way, consider the subtree containing $v_1$ after removing the edge incident to $v_1$ on the shortest path between $v_1,v_2$.
If this subtree does not intersect with $S-\{v_1\}$, then any boundary vertex on this subtree realizes the extreme point condition for $v_1$. Otherwise if the subtree intersects with $S-\{v_1\}$ at another vertex $v'_1\in S$, then consider further the subtree containing $v'_1$ by the same construction. Repeat this procedure until one finds a subtree that only intersects with $S$ at one vertex, and the procedure stops in finite steps since the graph is finite. Repeating this procedure from $v_2$ gives another extreme point, which verifies the Two-Points Condition.}
\end{examples}

For general cyclic graphs (i.e., graphs containing at least one cycle), it is often not easy to see if the Two-Points Condition is satisfied. The following proposition shows a concrete way to test for the Two-Points Condition.

\begin{proposition}\label{height}
For a finite graph with boundary $(G,\partial G,E)$, the Two-Points Condition follows from the existence of a function $h:G\cup\partial G\to \R$  satisfying the following conditions:
  \begin{enumerate}[(1)] \itemsep0em
  \item the Lipschitz constant of $h$ is bounded by $1$, i.e. if $x\sim y$ then $|h(x)-h(y)|\leqslant 1$;
  \item $|N_{\pm}(x)|=1$ for all $x\in G$, and $|N_{\pm}(z)|\leqslant
    1$ for all $z\in \partial G$, where
    \begin{eqnarray*}
    & &N_{+}(x)=\{y\in G\cup \p G:\ y\sim x \,,\, h(y)=h(x)+ 1\}, \\ 
    & &N_{-}(x)=\{y\in G\cup \p G:\ y\sim x \,,\, h(y)=h(x)- 1\}.
    \end{eqnarray*}
   \end{enumerate}
     We call $N_+(x)$ the discrete gradient of $h$ at $x$, and $N_-(x)$ the discrete gradient of $-h$ at $x$.
\end{proposition}

\begin{proof}
For any $S\subseteq G$ with $|S|\geqslant 2$, take the points where the function $h$ achieves maximum and minimum in $S$. Let $x_0\in S$ be any maximal point. By condition (2), we can take the unique path, denoted by $\gamma_{x_0}$, starting from $x_0$ such that each step increases the function $h$ by $1$. This path $\gamma_{x_0}$ can only pass each point of the graph at most once, and therefore the path must end somewhere since the whole graph is finite. Let $z_0$ be the point where the path $\gamma_{x_0}$ ends. By construction we know $|N_{+}(z_0)|=0$, which indicates $z_0\in \partial G$ by condition (2). Observe that $h(z_0)-h(x_0)\geqslant d(x_0,z_0)$ as $\gamma_{x_0}$ may not be distance-minimizing.

We claim that $x_0$ is the unique nearest point in $S$ from $z_0$ (i.e. $x_0$ is an extreme point of $S$). Suppose not, and there exists another $x_1\in S,\,x_1\neq x_0$ such that $d(x_1,z_0)\leqslant d(x_0,z_0)$. Then condition (1) implies that 
\begin{equation}\label{4.2claim}
|h(x_1)-h(z_0)|\leqslant d(x_1,z_0)\leqslant d(x_0,z_0).
\end{equation}
We claim that the two equalities in (\ref{4.2claim}) cannot hold at the same time. Suppose both equalities are achieved. We take the shortest path from $z_0$ to $x_1$, and then the length of this path is equal to $d(x_0,z_0)$ by the second equality. The function $h$ can only change $d(x_0,z_0)$ times along the shortest path from $z_0$ to $x_1$, and hence every change must be decreasing by $1$ in order to make both equalities hold. However, there can only exist one such path starting from $z_0$ due to condition (2), which is exactly the backward direction from $z_0$ to $x_0$. Along this path, $x_0$ would be reached in exactly $d(x_0,z_0)$ steps and hence $x_1=x_0$. Hence the equalities in (\ref{4.2claim}) cannot hold at the same time, and we have $h(x_1)>h(z_0)-d(x_0,z_0)$. On the other hand, we already know $h(x_0)\leqslant h(z_0)-d(x_0,z_0)$ indicating $h(x_1)>h(x_0)$, which is a contradiction to the maximality of $x_0$. This shows $x_0$ is an extreme point of $S$. 

The same argument shows that any minimal point is also an extreme point of $S$. Therefore the Two-Points Condition follows from the condition $|S|\geqslant 2$.
\end{proof}

\begin{examples}
{\rm Finite square, hexagonal (Figure \ref{Finite hexagonal lattice}, Left), triangular (Figure \ref{fig_ladder}, Left), graphite and square ladder (Figure \ref{fig_ladder}, Right) lattices satisfy the Two-Points Condition with the set of boundary vertices being the domain boundary. 

We can apply Proposition \ref{height} to verify this. For the square, hexagonal, graphite and square ladder lattices, the function $h$ can be chosen as the standard height function with respect to a proper floor. Note that for these lattices, it suffices to choose only the floor and the ceiling as the boundary. For triangular lattices, the function $h$ can be constructed as a group action, such that $h$ changes by $1/2$ along the horizontal direction and changes by $1$ along one of the other directions.}
\end{examples}


\begin{examples}
{\rm In the finite square, hexagonal, triangular, graphite and square ladder lattices, any horizontal edges can be removed 
and the obtained graphs still  satisfy the Two-Points Condition.
Here, the horizontal edges refer to the edges in the non-gradient directions with respect to the function $h$. See Figure \ref{Finite hexagonal lattice} (Right). This is because removing such edges does not affect the conditions for the function $h$ in Proposition \ref{height}. 

A finite square lattice with an interior vertex and all its edges removed also satisfies the Two-Points Condition. Essentially, one can repeat the proof of Proposition \ref{height} to show this particular situation. However, it is necessary to choose all four sides as the boundary and use two different choices of the function $h$. (Intuitively speaking, removing one small square does not affect the ability to find maximal and minimal points in at least two directions.) More generally, removing one square of any size from a finite square lattice does not affect the Two-Points Condition.}
\end{examples}

\begin{examples}
{\rm Assume that a function $h:G\cup\partial G\to \R$ satisfies  the conditions  (1,\,2) in Proposition \ref{height} for a finite graph with boundary $(G,\partial G,E)$. Then one can add to
the graph $(G,\partial G,E)$ additional edges $\{x,y\}$ that connect any two vertices $x,y\in G\cup\partial G$ satisfying $|h(x)-h(y)|<1$.
Similarly, one can remove from the graph any edges $\{x,y\}$ that connect vertices $x,y\in G\cup\partial G$ satisfying $|h(x)-h(y)|<1$.
 The obtained graph
and the function $h$  satisfy  the conditions  (1,\,2) in Proposition \ref{height}, and hence the  Two-Points Condition. This procedure can
 be used,  for example, to add additional horizontal edges in the finite hexagonal lattice in Figure \ref{Finite hexagonal lattice}.}
\end{examples}

\begin{examples}
{\rm Graphs that satisfy the conditions given in Proposition \ref{height} can be connected together so that these conditions stay valid.
To do this, assume that real-valued functions $h_1$ and $h_2$, defined on disjoint finite graphs with boundary 
$(G_1,\partial G_1,E_1)$ 
and
$(G_2,\partial G_2,E_2)$,
satisfy the conditions  (1,\,2) in Proposition \ref{height}, respectively. Moreover, assume that  there are $c\in \R$ and 
ordered sets $X_1=\{x_1^1,x_2^1,\dots,x_k^1\}\subseteq G_1\cup\partial G_1$ and 
$X_2=\{x_1^2,x_2^2,\dots,x_k^2\}\subseteq G_2\cup\partial G_2$, such that $|h_2(x^2_j)+c-h_1(x^1_j)|<1$ for all $j=1,2,\dots,k$. (In particular, such sets and $c$ always exist for $k=1$.)
Then we consider $(G,\partial G,E)$ for $G=G_1\cup G_2$, $\p G=\partial G_1\cup \partial G_2$ and $E=E_1\cup E_2\cup E_{12}$, where 
 $E_{12}=\big\{ \{x^1_j,x^2_j\}:\ j=1,2,\dots, k \big\}$. 
  We define a function $h$ on $G\cup\partial G$ by
  \begin{equation} \label{wSupport1}
h(x)=\begin{cases}
    h_1(x), &\hbox{for } x\in G_1\cup\partial G_1,\\
   h_2(x)+c,&\hbox{for } x\in G_2\cup\partial G_2.    \end{cases}
  \end{equation}
Then this function $h$ satisfies  the conditions (1,\,2) in Proposition \ref{height} for $(G,\partial G,E)$, and therefore the graph $(G,\partial G,E)$
satisfies  the Two-Points Condition. Figure \ref{fig_ladder} is a special case of this example.}
\end{examples}

However, the Two-Points Condition does not hold (without declaring more boundary vertices) if one adds an additional vertex and connects it to any interior vertex with an additional edge. This is because the set of the two endpoints of that additional edge violates the Two-Points Condition. 
Let us also mention an example where
the Two-Points Condition is not satisfied: the Kagome lattice.

\medskip
This paper is organized as follows. We introduce relevant definitions and basic facts in Section \ref{section:pre}. In Section \ref{section:wave}, we define the discrete wave equation and study the wavefront propagation. Section \ref{sectionspectral} is devoted to proving our main results.

\medskip
\noindent
{\bf Acknowledgement.}
We thank the anonymous referee for thoroughly reading our paper and many valuable comments.
H.I. was partially supported by Grant-in-Aid for Scientific Research (C)
20K03667 Japan
Society for the promotion of Science. H.I. is indebted to JSPS.
E.B., M.L.\ and J.L.\ were partially supported by Academy of Finland, grants 273979, 284715, 312110.

\section{Preliminaries}
\label{section:pre}

In this section, let $\mathbb G = (G,\partial G, E,\mu, g)$ be a finite weighted graph with boundary. First, we derive an elementary but important Green's formula.

\begin{lemma}[Green's formula]\label{Green}
  For two functions $u_1,u_2 : G\cup\partial G \to \R$, we have
  \[
  \sum_{x\in G}\mu_x \big( u_1(x) \Delta_G u_2(x) - u_2(x) \Delta_G
  u_1(x) \big) = \sum_{z\in\partial G} \mu_z \big( u_2(z) \partial_\nu
  u_1(z) - u_1(z) \partial_\nu u_2(z) \big).
  \]
\end{lemma}
\begin{proof}
By definition (\ref{Laplaciandef}),
\begin{eqnarray*}
\sum_{x\in G} \mu_x u_1(x)\Delta_G u_2(x) &=& \sum_{x\in G} \sum_{\substack{y\sim x \\ y\in G\cup\partial G}} g_{xy}u_1(x) \big(u_2(y)-u_2(x) \big) \\
&=& \sum_{x\in G} \big( \sum_{y\sim x,y\in G}+ \sum_{y\sim x,y\in \partial G} \big) g_{xy} u_1(x) \big(u_2(y)-u_2(x) \big).
\end{eqnarray*}
Observe that the indices and summations can be switched in the following way:
\begin{eqnarray*}
\sum_{x\in G} \sum_{y\sim x,y\in G} g_{xy} u_1(x)u_2(y)=\sum_{y\in G} \sum_{x\sim y,x\in G} g_{yx}u_1(y)u_2(x)=\sum_{x\sim y,x\in G}\sum_{y\in G} g_{yx}u_1(y)u_2(x).
\end{eqnarray*}
Hence the summation over $x,y\in G$ cancels out, and we get
\begin{eqnarray*}
\sum_{x\in G} \mu_x \big(u_1(x)\Delta_G u_2(x)- u_2(x)\Delta_G u_1(x)\big) =\sum_{x\in G}\sum_{y\sim x,y\in \partial G} g_{xy}\big(u_1(x)u_2(y)-u_2(x)u_1(y)\big),
\end{eqnarray*}
where we have used the fact that the weights are symmetric: $g_{xy}=g_{yx}$. Then the lemma follows from (\ref{Neumanndef}) and the following identity:
$$u_1(x)u_2(y)-u_2(x)u_1(y)=u_2(y)\big(u_1(x)-u_1(y)\big)+u_1(y) \big(u_2(y)-u_2(x)\big).$$
\vspace*{-10mm}

\end{proof}

Next we consider the boundary distance functions and the closely related
resolving sets of a graph, see \cite{resolving1,resolving2} and their generalizations in \cite{Hakanen2}. 

\begin{definition}\label{resolving}
  Let $(G,\partial G,E)$ be a finite connected graph with boundary. We say $\partial G=\{z_i\}_{i=0}^{m-1}$ is \emph{a resolving set} for $(G,\partial G,E)$,
  if the boundary distance coordinate 
  $$\big(d(\cdot,z_0), d(\cdot,z_1), \cdots, d(\cdot,z_{m-1})\big):G\to \mathbb{R}^m$$ 
  is injective, where $d$ denotes the distance on $(G\cup\partial G,E)$.

  For any point $x\in G$, we denote by $r_x$ the \emph{boundary distance 
  function}
  \begin{equation} \label{bndryDistFunc}
    r_x:\partial G \to \mathbb R, \qquad r_x(z)=d(x,z).
  \end{equation}
  The set of boundary distance functions of a graph $\mathbb G$ is
  denoted by $\mathcal{R}(\mathbb G)$. If $\partial G$ is a resolving
  set, the map $x\mapsto r_x$ from $G$ to $\mathcal{R}(\mathbb G)$ is
  a bijection.
\end{definition}

The minimal cardinality of resolving sets for a graph is called the metric dimension of the graph \cite{resolving1}.
The boundary distance functions are extensively used in the study of inverse problems on manifolds, see e.g.\ \cite{HoopS,Ivanov,KKL,LSaksala,PS}.
 
The concept of resolving sets gives a rough idea on how to choose boundary points such that the inverse problem may be solvable. If the chosen boundary points do not form a resolving set, then there is little hope to solve the inverse problem from spectral data measured at those points. 

\begin{lemma}\label{TPA-resolving}
The Two-Points Condition (\cref{as1case1} of
\cref{assumption}) implies that $\partial G$ is a resolving set
for $(G,\partial G,E)$.
\end{lemma}
\begin{proof}
Suppose that $\partial G$ is not a
resolving set for $(G,\partial G,E)$. Then by \cref{resolving}, there exist two points
$x_1,x_2\in G$ such that $d(x_1,z)=d(x_2,z)$ for all $z\in \partial
G$. However, the set $S=\{x_1,x_2\}$ is a contradiction to the Two-Points
Condition for $(G,\partial G,E)$. 
\end{proof}

We point out that the Neumann spectral data for the equation (\ref{eigenvalueproblem}) are not affected at all by edges between boundary points, since the Neumann boundary value (\ref{Neumanndef}) only counts edges from boundary points to interior points. In other words, the edges between boundary points are invisible to our Neumann spectral data. However, this limitation does not matter to us since the structure of the boundary is \emph{a priori} given. What we will reconstruct in the next few sections is actually the reduced graph of $\mathbb{G}$, which is defined as follows.
\begin{definition}[Reduced graph]\label{reduceddef}
  Let $\mathbb G = (G,\partial G,E,\mu,g)$ be a weighted graph with boundary. 
  The \emph{reduced graph} of $\mathbb{G}$ is defined as 
  $\mathbb G_{re} = (G,\partial G,E_{re},\mu,g|_{E_{re}} ),$
    where
  \[
  E_{re} = E - \big\{ \{x,y\} \in E \,\big|\, x\in\partial G \text{ and }
  y\in\partial G \big\}.
  \]
A graph with boundary being strongly connected is equivalent to its reduced graph being connected.
Note that $\mathbb{G}$
and $\mathbb G_{re}$ have identical Neumann spectral data due to the definition of the Neumann boundary value (\ref{Neumanndef}).
\end{definition}

Reducing a graph affects distances as paths along edges between boundary points become
forbidden. In the same way as \cref{distance}, the distance $d_{re}(x,y)$
on the reduced graph $(G\cup \partial G,E_{re})$ is defined through
paths of $(G\cup\partial G,E_{re})$ from $x$ to $y$, instead of along
paths of the original graph $(G\cup\partial G, E)$. Then clearly $d_{re}(x,y) \geqslant d(x,y)$ for any $x,y\in G\cup\partial G$. 
The change of distances also affects the $r$-neighbourhood
$N_{re}(x,r)$ of $x\in G\cup\partial G$, which is defined by
\[
N_{re}(x,r)=\big\{y\in G\cup\partial G\,\big|\, d_{re}(y,x)\leqslant r \big\}.
\]

However, reducing a graph does not affect the Two-Points Condition.

\begin{lemma}\label{reduced}
  If $x_0\in S$ is an extreme point of a subset $S\subseteq G$ realized by
  some $z\in \partial G$, then there exists $z_0\in \partial G$ also
  realizing the extreme point condition of $x_0$ such that none of the
  shortest paths from $x_0$ to $z_0$ pass through any other boundary
  point.
  
  As a consequence, if $(G,\partial G,E)$ satisfies the Two-Points
  Condition (\cref{as1case1} of \cref{assumption}), then so does
  $(G,\partial G,E_{re})$ with respect to its distance function
  $d_{re}$.
\end{lemma}
\begin{proof}
  If any shortest path from $x_0$ to $z$ passes through another
  boundary point $z^{\prime}\in \partial G$, then $x_0$ is an extreme
  point also realized at $z^{\prime}$. Then we consider the set of all
  the boundary points with respect to which $x_0$ is an extreme point,
  and take a point $z_0$ (not necessarily unique) in this set with the
  minimal distance from $x_0$. It follows that $z_0$ is the desired
  boundary point; otherwise there is another boundary point in the set
  with a smaller distance from $x_0$.

  Let $x_0\in S$ be an extreme point of $S$ with respect to
  $\partial G$ realized by $z_0\in\partial G$. By the argument above, we may assume that none of the shortest paths from $x_0$
  to $z_0$ pass through any boundary point except for $z_0$. Reducing
  the graph will not affect this path or its length. On the other
  hand, no distances between points may decrease in the
  reduction. So this path is still the shortest path between $S$ and
  $z_0$ in the reduced graph. Hence $x_0$ is also an extreme point
  of $S$ with respect to $\partial G$ in the reduced graph.
\end{proof}

\section{Wave Equation}
\label{section:wave}

\begin{definition}[Time derivatives]
  For a function $u: G\times \mathbb{N}\to \mathbb{R}$, we define the discrete first and second time derivatives at $(x,t_0)$ by
  \ba
& &  D_t u(x,t_0)=u(x,t_0+1)-u(x,t_0), \qquad t_0 \geqslant 0,
  \\
& &  
  D_{tt} (x,t_0)=u(x,t_0+1)-2u(x,t_0)+u(x,t_0-1), \qquad t_0 \geqslant 1.
  \ea
  These are sometimes called the forward difference and the
  second-order central difference in time.
\end{definition}

We consider the following initial value problem for the discrete wave
equation with the Neumann boundary condition:
\begin{equation}\label{waveequation}
    \begin{cases}
    D_{tt}  u(x,t)-\Delta_G u(x,t)+q(x)u(x,t)=0, &x\in G,\, t\geqslant 1,\\
      \partial_{\nu} u(x,t)=0, &x\in\partial G,\, t\geqslant 0,\\
      D_t u(x,0) = 0, &x\in G,\\
      u(x,0) = W(x), &x\in G\cup \partial G,
    \end{cases}
\end{equation}
where the values of $u$ on $\partial G$ are uniquely determined by
the values on $G$ via the Neumann boundary condition at each time step. More precisely, using the definition of the Neumann boundary value (\ref{Neumanndef}) gives
\begin{equation} \label{uValuesOnBndry}
  u(z) = \sum_{\substack{x\sim z\\x\in G}} g_{xz} u(x) \Big/
  \sum_{\substack{x\sim z\\x\in G}} g_{xz}, \qquad z\in\partial G.
\end{equation}
We require $\partial_{\nu}W|_{\partial G}=0$ for the compatibility of
the initial value and Neumann boundary condition. The initial
conditions and the boundary condition imply that $u(x,1)=u(x,0)=W(x)$
for all $x\in G\cup\partial G$.

\begin{definition}[Waves]
  Given $W : G \cup \partial G \to \R$ satisfying $\partial_\nu W|_{\partial G} = 0$, denote by $u^W:(G\cup\partial G)\times
  \mathbb{N}\to \mathbb{R}$ the solution of the discrete wave equation
  \eqref{waveequation} with the initial condition
  $u(\cdot,0)=W(\cdot)$ on $G\cup\partial G$. The function $W$ is
  called the \emph{initial value}. In this paper, a \emph{wave} refers
  to a solution of the wave equation \eqref{waveequation}.
\end{definition}

\begin{lemma}
Given any initial value $W : G \cup \partial G \to \R$ satisfying $\partial_\nu W|_{\partial G} = 0$,
the discrete wave equation \eqref{waveequation} has a unique solution.
\end{lemma}
\begin{proof}
The discrete wave equation is solved in the following way.
The solution on $G\cup\partial G$ at times $t=0$
and $t=1$ are determined by the initial conditions. Afterwards, the value on $G$ at time $t\geqslant 2$ is
calculated from the value on $(G\cup\partial G) \times \{t-1\}$ and on
$G\times\{t-2\}$ by the equation $D_{tt}  u - \Delta_G u + qu=0$.
Then the formula \eqref{uValuesOnBndry} gives the value on $\partial G$ at time
$t$.
\end{proof}

\smallskip
The main purpose of this section is to prove a wavefront lemma which will be used frequently in
the next section. Item~\ref{as1case2} of \cref{assumption} is
essential for the wavefront lemma, as the wave propagation may ``speed
up'' due to the instantaneous effect of the boundary condition if a
shortest path goes through the boundary. Under \cref{as1case2} of \cref{assumption}, distances of the reduced graph
are realized by avoiding boundary points, which is essential to guarantee proper wave behaviour.
\begin{lemma} \label{avoidBndry}
  Let $\mathbb G$ be a finite connected graph with boundary satisfying
  \cref{as1case2} of \cref{assumption}.
  Suppose the reduced graph of $\mathbb G$ is connected. 
  Let $x\in G$ and $z\in\partial
  G$. If $x\sim z$, then $d_{re}(x,p) \leqslant d_{re}(z,p)$ for any $p\in
  G\cup\partial G - \{z\}$.
\end{lemma}
\begin{proof}
  Let $x'\in G$ be a point such
  that $x'\sim z$ and $d_{re}(x',p) = d_{re}(z,p)-1$. 
  This point exists since distances are realised by paths in a connected graph.
  Such a point $x'$
  cannot be in $\partial G$, because there are no edges between
  boundary points in the reduced graph.
  We have $x,x'\in G$ and $x\sim z$, $x'\sim z$. Then by
  \cref{as1case2} of \cref{assumption}, we have $x\sim x'$ if $x\neq x'$. Hence the
  triangle inequality yields that $d_{re}(x,p) \leqslant d_{re}(x,x') + d_{re}(x',p) =
  d_{re}(z,p)$. If $x=x'$, then $d_{re}(x,p)=d_{re}(z,p)-1<d_{re}(z,p)$.
\end{proof}

Recall that the Neumann boundary value
\eqref{Neumanndef} does not take into account the edges between
boundary points. This means that waves cannot propagate from one
boundary point to another without going through the interior.
Hence the wavefront propagates by the distance function
$d_{re}$ of the reduced graph, instead of the distance
function of the original graph.
\begin{lemma}[Wavefront]\label{wavefront}
  Let $\mathbb G$ be a finite connected weighted graph with boundary satisfying
  \cref{as1case2} of \cref{assumption}.
  Suppose the reduced graph of $\mathbb G$ is connected. Let $x\in G,\, z\in\partial G$ and $t_0=d_{re}(x,z)$.
  Suppose $W:G\cup\partial G\to\R$ is an initial value satisfying
  $\partial_{\nu} W|_{\partial G}=0,\,W(z)=0$ and
  \begin{equation} \label{wSupportInX}
    \big\{y\in N_{re}(z,t_0)\cap G \;\big|\; W(y)\neq 0 \big\}\subseteq \{x\}.
  \end{equation}
  Then the following properties hold for the wave $u^W$ generated by $W$:
  \begin{enumerate}[\quad(1)]
  \item If $W(x)> 0$, then $t_0\geqslant 2$, $u^W(z,t_0)> 0$ and
    $u^W(z,t)= 0$ for all $t<t_0$;
  \item If $W(x)=0$, then $u^W(z,t)=0$ for all $t\leqslant t_0$.
  \end{enumerate}
\end{lemma}
\begin{proof}
  Let us prove the first claim of the lemma with $W(x)>0$. To start with, we show that
  $t_0=d_{re}(x,z)\geqslant2$. Suppose $d_{re}(x,z)=1$. Let
  $x,y_1,\ldots,y_J \in G$ be the interior points connected to $z\in\partial G$. The
  boundary conditions $W(z)=0$ and $\partial_\nu W(z)=0$ imply that
  \[
  \sum_{\substack{y\sim z\\y\in G}} g_{yz}W(y) = 0.
  \]
  On the other hand, \eqref{wSupportInX} implies that $W(y_j)=0$ for
  $j=1,\ldots,J$. Hence the equation above reduces to $g_{xz}W(x)=0$,
  which is a contradiction as $g$ is defined to be positive. Hence
  $d_{re}(x,z)\geqslant 2$.

  Let $ z'\in\partial G$, $z'\neq z$. If
  $d_{re}(z',z) \leqslant t_0-1$, then for any $x'\in G$ with $x'\sim z'$,
  \cref{avoidBndry} implies that $d_{re}(x',z) \leqslant d_{re}(z',z) \leqslant
  t_0-1$, and hence $W(x')=0$ by \eqref{wSupportInX}. 
  Since this holds for all the interior points connected to $z'$, the
  Neumann boundary condition yields that $W(z')=0$. 
  On the other hand, if $d_{re}(z',z) =
  t_0$, then for any $x'\in G$ with $x'\sim z'$,
  \cref{avoidBndry} implies that $d_{re}(x',z)\leqslant t_0$, and hence
  $W(x')\geqslant 0$.
  Note that $W(x')$ may be nonzero in this case since possibly $x'=x$.
   Then the Neumann boundary
  condition gives $W(z')\geqslant 0$. Combining these observations, for any $p\in G \cup \partial G$, we have
  \begin{equation} \label{wSupport2}
    W(p)\,\begin{cases}
    = 0, & p\in G\cup\partial G,\, d_{re}(p,z) \leqslant t_0-1,\\
    \geqslant 0, & p\in G\cup\partial G,\, d_{re}(p,z) = t_0,\\
    > 0, & p=x.
    \end{cases}
  \end{equation}
By the
  initial conditions of the wave equation \eqref{waveequation}, $u^W(p,1)=u^W(p,0)=W(p)$ for all $p\in G
  \cup \partial G$.
Hence \eqref{wSupport2} gives the wavefront behavior of the wave $u^W(\cdot,t)$ at time $t=0,1$.
  
\smallskip
To study the wavefront behavior of the wave $u^W(\cdot,t)$ at $t\geqslant 2$, it is convenient to use an induction formulation where \eqref{wSupport2} serves as the base case.  
The formulation is as follows: prove the following statement, by induction on $\tau = 1,2,\ldots, t_0-1$, that for any $p\in G \cup \partial G$,
  \begin{equation} \label{indClaim1}
    u^W(p,\tau)\, \begin{cases}
      = 0, &p\in G\cup\partial G,\, d_{re}(p,z)\leqslant t_0-\tau,\\
      \geqslant 0, &p \in G\cup\partial G,\, d_{re}(p,z) = t_0-\tau+1,
    \end{cases}
  \end{equation}
   and that
  \begin{equation} \label{indClaim2}
    u^W(x_\tau,\tau) > 0,
  \end{equation}
  for some $x_\tau \in G$ satisfying $d_{re}(x_\tau,z) = t_0-\tau+1$. For $\tau =1$, the claims \eqref{indClaim1} and \eqref{indClaim2} reduce to \eqref{wSupport2}
  by choosing
  $x_\tau=x$. This verifies the initial conditions for the induction. Assume that \eqref{indClaim1} and \eqref{indClaim2} hold
  for some $\tau\in\{1,2,\ldots,t_0-2\}$, we need to prove that \eqref{indClaim1} and \eqref{indClaim2} hold for $\tau+1$. We will spend most of the proof to argue this. Once \eqref{indClaim1} and \eqref{indClaim2} are proved, we will show in the end that the lemma can be proved from the $\tau=t_0-1$ case.
  
  \smallskip
  By the wave equation \eqref{waveequation}, we have
  \begin{equation} \label{indUpdate}
    u^W(p,\tau+1) = 2u^W(p,\tau) - u^W(p,\tau-1) + \Delta_G u^W(p,\tau) - q(p)
    u^W(p,\tau),
  \end{equation}
  when $p\in G$ and $\tau\geqslant 1$. This formula and the Neumann boundary condition are what the induction is based on. 
  First, we prove that \eqref{indClaim1} holds for $\tau+1$.

  Let $p\in G$ satisfying $d_{re}(p,z) \leqslant t_0-(\tau+1)$. Then we see that the terms
  $2u^W(p,\tau), u^W(p,\tau-1)$ and $q(p)u^W(p,\tau)$ in \eqref{indUpdate} are all equal to
  zero by the induction assumption. Moreover, since $u^W(p,\tau)=0$, we have
  \[
  \Delta_Gu^W(p,\tau)= \frac{1}{\mu_p} \sum_{\substack{y\sim p\\y\in
      G\cup\partial G}} g_{py} u^W(y,\tau).
  \]
  Let $y\in G\cup\partial G$
  be any point connected to $p$. Then $d_{re}(y,z) \leqslant d_{re}(y,p) +
  d_{re}(p,z) \leqslant  1 + t_0 - (\tau+1) = t_0-\tau$, and hence $u^W(y,\tau)=0$
  by the induction assumption. Thus \eqref{indUpdate} shows that
  $u^W(p,\tau+1) = 0$ for all $p\in G$, $d_{re}(p,z)\leqslant
  t_0-(\tau+1)$.

  On the other hand, if $p\in G$ satisfying $d_{re}(p,z)=t_0-(\tau+1)+1 =
  t_0-\tau$, then for the same reason as above, we see that
  \[
  u^W(p,\tau+1) = \Delta_G u^W(p,\tau)= \frac{1}{\mu_p}
  \sum_{\substack{y\sim p\\y\in G\cup\partial G}} g_{py} u^W(y,\tau).
  \]
  If $y\in G\cup\partial G$ satisfies $y\sim p$, then $d_{re}(y,z) \leqslant
  d_{re}(y,p) + d_{re}(p,z) = 1 + t_0 - \tau$, and hence $u^W(y,\tau) \geqslant 0$ by the induction assumption. Thus
  $u^W(p,\tau+1)\geqslant 0$. It remains to consider the case of $p\in\partial
  G$, and find $x_{\tau+1}$ for \eqref{indClaim2}.

  Let $p\in\partial G$. Instead of using \eqref{indUpdate} which is valid
  only in the interior, we can determine the sign of
  $u^W(p,\tau+1)$ by using the Neumann boundary condition $\partial_\nu
  u^W(p,\tau+1)=0$. Namely,
  \begin{equation} \label{indUpdateNeumann}
    u^W(p,\tau+1) = \sum_{\substack{y\sim p\\y\in G}} g_{yp}
    u^W(y,\tau+1) \Big/ {\sum_{\substack{y\sim p\\y\in G}} 
      g_{yp}}.
  \end{equation}
  Suppose $p\neq z$ and $d_{re}(p,z)\leqslant t_0-(\tau+1)$. 
  Any interior point $y$ with $y\sim p$ satisfies that
  $d_{re}(y,z) \leqslant d_{re}(p,z) \leqslant t_0 - (\tau+1)$ by
  \cref{avoidBndry}. Since we have already showed that $u^W(y,\tau+1) = 0$ for
  any $x\in G$ satisfying $d_{re}(y,z)\leqslant t_0-(\tau+1)$, it follows from \eqref{indUpdateNeumann} that
  $u^W(p,\tau+1)=0$. In the case of $p=z$, for any interior point $y$ adjacent to $z$ (i.e., $d_{re}(y,z)=1$),
  we have
  $u^W(y,\tau+1)=0$ if $d_{re}(y,z)=1\leqslant t_0-(\tau+1)$ is satisfied. This is applicable to all our induction steps since $\tau\leqslant t_0-2$, and therefore we have $u^W(z,\tau+1)=0$ due to \eqref{indUpdateNeumann}.

  For the second line in \eqref{indClaim1}, 
  let $p\in \partial G$ satisfying $d_{re}(p,z)= t_0-(\tau+1)+1 = t_0-\tau$. 
  In particular $p\neq z$ since $\tau<t_0$. 
  As in the previous case, we
  see that any interior point $x$ with $x\sim p$ satisfies $d_{re}(x,z)
  \leqslant d_{re}(p,z) = t_0-\tau$. Since we have already showed that
  $u^W(x,\tau+1)\geqslant 0$ for such $x$, we get $u^W(p,\tau+1)\geqslant 0$. 
  This concludes the proof of \eqref{indClaim1} by induction.

\smallskip
  Next, we prove that \eqref{indClaim2} holds for $\tau+1$. The
  induction assumption gives that $u^W(p,\tau)\geqslant 0$ for any $p\in
  G\cup\partial G$ satisfying $d_{re}(p,z)=t_0-\tau+1$. Moreover,
  there exists one
  such $p\in G$, denoted by $x_\tau$, so that $u^W(x_\tau,\tau)>0$. Let
  $\gamma$ be a shortest path of length $t_0-\tau+1$ from $x_\tau$ to $z$ in
  the reduced graph. Since $\tau\leqslant t_0-2$, this path is at least of
  length $3$. Let $x_{\tau+1}$ be the second vertex along this path,
  and then $d_{re}(x_{\tau+1},z) = t_0-\tau\geqslant 2$. Observe that $x_{\tau+1}$ is also an
  interior point: if not, then \cref{avoidBndry} implies that
  $d_{re}(x_\tau,z) \leqslant d_{re}(x_{\tau+1},z) = t_0-\tau$ as $x_\tau \sim x_{\tau+1}$, contradiction.
  
  To prove \eqref{indClaim2}, it remains to prove that
  $u^W(x_{\tau+1},\tau+1)>0$. We consider the formula \eqref{indUpdate} with
  $p=x_{\tau+1}$. The induction assumption for \eqref{indClaim1} shows
  that $u^W(x_{\tau+1},\tau)$, $u^W(x_{\tau+1},\tau-1)$ and
  $q(x_{\tau+1}) u^W(x_{\tau+1},\tau)$ are all equal to zero, since
  $d_{re}(x_{\tau+1},z) \leqslant t_0-\tau$. Thus by \eqref{indUpdate},
  \[
  u^W(x_{\tau+1},\tau+1) = \Delta_G u^W(x_{\tau+1},\tau)=
  \frac{1}{\mu_{x_{\tau+1}}} \sum_{\substack{y\sim x_{\tau+1}\\y\in
      G\cup\partial G}} g_{yx_{\tau+1}} u^W(y,\tau).
  \]
  For a point $y\in G\cup\partial G$ connected to $x_{\tau+1}$, we have $d_{re}(y,z) \leqslant t_0 -
  \tau + 1$, and therefore the induction assumption for \eqref{indClaim1} gives
  $u^W(y,\tau)\geqslant 0$. Notice that one of the points $y$ in the sum above is
  $x_\tau$, for which $u^W(x_\tau,\tau)>0$. Hence the whole sum is
  positive. This concludes the proof of 
  \eqref{indClaim2} by induction.

  \medskip
  Now we turn to the statement of the lemma, with \eqref{indClaim1} and \eqref{indClaim2} in hand.
  We see that $u^W(z,t)=0$ for all $t<t_0$ by \eqref{indClaim1}. At
  time $t_0$, the Neumann boundary condition gives
  \begin{equation} \label{zValue}
    u^W(z,t_0) = \sum_{\substack{y\sim z\\y\in G}} g_{yz} u^W(y,t_0)
    \Big/ {\sum_{\substack{y\sim z\\y\in G}} g_{yz}}.
  \end{equation}
  Let $y\in G$ be an arbitrary point satisfying $y\sim z$. The formula \eqref{indUpdate} gives
  \[
  u^W(y,t_0) = 2u^W(y,t_0-1) - u^W(y,t_0-2) + \Delta_G u^W(y,t_0-1) -
  q(y) u^W(y,t_0-1).
  \]
  Since $u^W(y,t_0-1)=u^W(y,t_0-2)=0$ by
  \eqref{indClaim1}, we have
  \[
  \Delta_G u^W(y,t_0-1) = \frac{1}{\mu_y} \sum_{\substack{p\sim
      y\\p\in G\cup\partial G}} g_{yp} u^W(p,t_0-1).
  \]
  Note that $d_{re}(p,z) \leqslant d_{re}(p,y) + d_{re}(y,z) = 1+1 = 2$. If
  $d_{re}(p,z)=0$, then $p=z$ and 
  $u^W(p,t_0-1)=u^W(z,t_0-1)=0$. If $d_{re}(p,z)=1$, then
  $u^W(p,t_0-1)=0$ by the first line of \eqref{indClaim1}. If
  $d_{re}(p,z)=2$, then $u^W(p,t_0-1)\geqslant 0$ by the second line of \eqref{indClaim1}.
Hence from the equations above, we see that $\Delta_G
  u^W(y,t_0-1)\geqslant 0$ and consequently $u^W(y,t_0)\geqslant 0$ for any $y\in G$ satisfying $y\sim z$.  
  Furthermore, there exists a point $\hat{p}\in G$ with $d_{re}(\hat{p},z)=2$ such that $u^W(\hat{p},t_0-1)>0$ by
  \eqref{indClaim2}. 
The condition $d_{re}(\hat{p},z)=2$ indicates that there exists a point $\hat{y}\in G$ such that $\hat{y}\sim z$ and $\hat{y}\sim \hat{p}$.  
  Hence from the same equations above, we see that $\Delta_G
  u^W(\hat{y},t_0-1)>0$ and consequently $u^W(\hat{y},t_0)>0$.
Combining these with \eqref{zValue} yields that $u^W(z,t_0)>0$.

  \smallskip
  The second claim of the lemma with $W(x)=0$ simply follows from the same proof as above
  but without the need for \eqref{indClaim2}, and by replacing
  instances of $W\geqslant 0, W>0$ with $W=0$ and those of $u^W\geqslant 0, u^W>0$
  with $u^W=0$.
\end{proof}

\section{The Inverse Spectral Problem}\label{sectionspectral}

In this section, we reconstruct the graph structure and the potential from the Neumann
boundary spectral data, and prove \cref{structureThm,coefficientThm}. Since the structure of the boundary is \emph{a priori} given, it suffices to reconstruct the reduced graph $\mathbb{G}_{re}$ (recall
\cref{reduceddef}). The assumption that $\mathbb{G}$ is strongly connected is equivalent to $\mathbb{G}_{re}$ being connected. Due to \cref{reduced} and the fact that removing edges between
boundary points does not affect the boundary spectral data, without loss of generality, \emph{we assume $\mathbb G = \mathbb
G_{re}$ throughout this section}. In other words, we assume that there are no
edges between boundary points in $\mathbb G$. 

\smallskip
The full reconstruction process is divided into two main parts. 
The first part proves that under the Two-Points Condition, the Neumann boundary spectral data determine the Fourier coefficients of all $L^2$-normalized functions supported at one single interior point (and all boundary distance functions corresponding to interior points). The second part proves that these information then determines the graph structure and the weights.

\subsection{Characterization by boundary data} 
\label{characterizationObjects}

In this subsection, we construct a characterization of the boundary distance functions by boundary data.
We mention that the related constructions on partially ordered lattices that contain boundary distance functions as maximal elements have been used to study inverse problems on manifolds in \cite{Oksanen}.

\smallskip
Let $s : \partial G \to \Z_+$. We equip the set of such functions with the following partial order
\begin{equation} \label{partOrder}
  \forall z\in\partial G : s_1(z)\leqslant s_2(z) \quad
  \Longrightarrow \quad s_1\leqslant s_2.
\end{equation}
We consider the set of initial values for which the corresponding waves are not
observed at the boundary before time $s(\cdot)$,
\[
\mathcal{W}(s)=\{W: G\cup\partial G\to \R \;| \;
\partial_{\nu}W|_{\partial G}=0,\, u^W(z,t)=0 \textrm{ for all }z\in
\partial G,\, t<s(z)\}.
\]
Let $N = |G|$ be the number of interior vertices, and we define the set
\begin{equation}
  \mathcal{U} = \{ s:\partial G \to \Z_+ \mid 2\leqslant
  s(\cdot)\leqslant N,\, \dim(\mathcal{W}(s))\neq 0\}.
\end{equation}
We remark that given the Neumann boundary spectral data, the conditions of $\mathcal{W}(s)$ correspond to a system of linear equations on $\partial G$ for solving the Fourier coefficients of the initial value, which will be explained in details in the next subsection. As a consequence, the Neumann boundary spectral data determine the set $\mathcal{U}$.

Due to the linearity of the wave equation \eqref{waveequation}, the set $\mathcal W(s)$ is a linear space over $\R$, so $\dim(\mathcal
W(s))$ is simply its dimension as a vector space. The condition $\dim(\mathcal
W(s))\neq 0$ simply means that there exists a nonzero initial value such that the corresponding wave satisfies the conditions of $\mathcal{W}(s)$.
Observe that the conditions of $\mathcal{W}(s)$ indicate that any
initial value $W\in \mathcal{W}(s)$ vanishes on the whole boundary
since $u^W(z,0)=W(z)$. Then the condition $u^W(z,t)=0$ for $t=0$
implies the same condition for $t= 1$ due to the initial conditions of
\eqref{waveequation}. Hence we only need to consider $s\geqslant 2$ in the definition above.


\medskip
The set $\mathcal U$ is a set of functions equipped with the
partial order \eqref{partOrder}. We are interested in its maximal elements with respect to the partial order, denoted by $\max(\mathcal{U})$.

\begin{lemma}\label{maximal}
  Let $\mathbb G$ be a finite connected weighted graph with boundary satisfying \cref{assumption}.
   Then we have
  \begin{equation}
    \{r_x \mid x\in G - N(\partial G)\} \subseteq \max(\mathcal U).
  \end{equation}
  Furthermore, for any nonzero initial value $W\in \mathcal{W}(r_x)$
  where $x\in G-N(\partial G)$, we have $\supp(W)=\{x\}$. \footnote{Precisely, $r_x$ is the boundary distance function with respect to $d_{re}$ on the reduced graph $\mathbb{G}_{re}$. Recall that we assumed $\mathbb G = \mathbb
G_{re}$. This is the case throughout Section \ref{sectionspectral}.}
\end{lemma}
\begin{proof}
  For an arbitrary point $x\in G-N(\partial G)$, we first show that
  for any nonzero initial value $W\in \mathcal{W}(r_x)$, we have
  $\supp(W)=\{x\}$ and consequently $r_{x}\in \mathcal{U}$ by
  \cref{wavefront}.
  Observe that $r_x(\cdot)\geqslant 2$ due to $x\in G-N(\partial G)$, and $r_x(\cdot)\leqslant N$ since we assume no edges between boundary points and any shortest path between $x$ and a boundary point can pass at most $N-1$ interior points.

  From the condition $W\in \mathcal{W}(r_{x})$, we see that the wave
  $u^W$ corresponding to this initial value satisfies $u^W(z,t)=0$ when
  $z\in\partial G$ and $t<r_{x}(z)$. Let
  \[
  S=\{y\in G \mid W(y)\neq 0\}\cup\{x\}.
  \]
  If $ |S|\geqslant 2$, the Two-Points Condition in \cref{assumption}
  implies that there exist $x_0\in S-\{x\}$ and $z_0\in \partial G$,
  such that $x_0$ is the unique nearest point in $S$ from $z_0$, which
  in particular yields $d(x_0,z_0)\leqslant d(x,z_0)-1 = r_x(z_0)-1$.
  But by the propagation of the wavefront (\cref{wavefront}), we see
  that $u^W(z_0,t_0)\neq 0$ for $t_0=d(x_0,z_0)$. This is a
  contradiction because $\mathcal{W}(r_x)$ requires
  $u^W(z_0,t_0)=0$ for $t_0<r_x(z_0)$. Therefore $|S|\leqslant 1$ and
  $\supp(W) \cap G\subseteq \{x\}$. Since
  $x\notin N(\partial G)$, we see by the Neumann boundary condition that $W=0$
  on $\partial G$. Hence $\supp(W) = \{x\}$.

  We next show that the boundary distance functions are maximal
  elements in $\mathcal U$. Let $x\in G-N(\partial G)$ and suppose
  there exists an element $s\in \mathcal{U}$ such that $s\geqslant
  r_x$. By definition of $\mathcal{U}$, there exists a nonzero initial
  value $W$ such that
  \begin{equation} \label{sVanishingProp}
    u^W(z,t)=0, \qquad \forall z\in \partial G,\,t<s(z).
  \end{equation}
  Since $r_{x}\leqslant s$, the same vanishing conditions hold for all
  $t<r_{x}(z)$. Then the same argument above yields $\supp(W)= \{x\}$. If
  $r_{x}(z')<s(z')$ for some boundary point $z'\in \partial G$ then
  \cref{wavefront} shows that $u^W(z',t')\neq 0$ for
  $t'=d(x,z')=r_x(z')$. This contradicts \eqref{sVanishingProp}.
  Therefore if $r_{x}\leqslant s$, then $r_{x}=s$ and therefore $r_x$ is a
  maximal element in $\mathcal U$.
\end{proof}


Next, we recover the boundary distance functions corresponding to points in $G\cap
N(\partial G)$. 
Let $s : \partial G \to \Z_+$, and we define
\begin{align*}
  \mathcal{W}_b(s)=\{W:G\cup\partial G\to \R \mid
  \partial_{\nu}W|_{\partial G}=0,\, u^W(z,t)=0 \textrm{ when }z\in
  \partial G,\, s(z)\geqslant 2&\\ \textrm{and } t<s(z)\}&,
\end{align*}
and for $y\in G\cap N(\partial G)$, define the set
\[
\mathcal{U}_b(y) = \{s:\partial G\to \Z_+ \mid
\dim(\mathcal{W}_b(s))\neq 0,\,s(\cdot)\leqslant N,\, s(z)=1 \textrm{
  only if }z\sim y\}.
\]
Recall that $N=|G|$ is the number of interior
points. Functions $s\in\mathcal U_b(y)$ can have $s(z)>1$ at $z \sim
y$. 
As with the previous case, the set $\mathcal{U}_b(y)$ is also determined by the Neumann boundary spectral data, which will be explained in the next subsection.

One can show the following lemma by a similar argument as
\cref{maximal}.
\begin{lemma}\label{maximalb}
  Let $\mathbb G$ be a finite connected weighted graph with boundary satisfying \cref{assumption}.
  Then for any $y\in G\cap N(\partial G)$, we have
  \begin{equation}
    r_y \in \max\big(\mathcal U_b(y)\big).
  \end{equation}
  Furthermore, for any nonzero
  initial value $W\in \mathcal{W}_b(r_y)$ where $y\in G\cap N(\partial
  G)$, we have $supp(W)\cap G=\{y\}$.
\end{lemma}
\begin{proof}
  Following the argument in \cref{maximal}, for any nonzero initial
  value $W\in \mathcal{W}_b(r_y)$, consider the set $S=\{x\in G \mid
  W(x)\neq 0\}\cup\{y\}$. If $|S|\geqslant2$, we can find an
  extreme point $y_0\in G - \{y\}$ of $S$ with respect to some $z_0\in
  \partial G$. The extreme point condition implies that $z_0$ cannot
  be connected to $y$, and hence $W(z_0)=u^W(z_0,0)=0$ by the
  condition $W\in\mathcal W_b(r_y)$. The assumptions of
  \cref{wavefront} are satisfied for $W$ and the pair of points
  $y_0,z_0$, so $u^W(z_0,t_0)\neq 0$ for $t_0 = d(y_0,z_0)$. But by
  the definition of the extreme point, we have $d(y_0,z_0) < d(y,z_0)
  = r_y(z_0)$. This contradicts the condition $u^W(z_0,t)=0$ for $t<r_y(z_0)$
  of $W\in \mathcal W_b(r_y)$, considering $z_0 \not\sim
  y$. Hence $|S|=1$ and $\supp(W)\cap G= \{y\}$.

  Let $y\in G\cap N(\partial G)$. We first show that
  $r_y\in\mathcal U_b(y)$. Clearly $r_y\leqslant N$ and $r_y(z)=1$
  only at boundary points $z$ connected to $y$. It remains to show
  that there exists a nonzero initial value in $\mathcal W_b(r_y)$. Consider an initial value
  $W$ satisfying $W(x)=1$ at $x=y$ and $W(x)=0$ otherwise in $G$. The values of $W$ on $\partial G$ are determined by 
  the Neumann boundary condition (\ref{uValuesOnBndry}). 
  By the definition of $W_b(r_y)$, it suffices to show that $u^W(z,t)=0$ for all
  $t<r_y(z)$ when $r_y(z)\geqslant 2$. At such boundary points $z$ satisfying $r_y(z)\geqslant 2$ (i.e. $z \not\sim y$), the Neumann condition gives
  $W(z)=0$. Moreover, we have $W=0$ at all points in $N(z,d(y,z)) \cap G$
  except for $y$ at which $W(y)>0$. Hence \cref{wavefront} yields
  that $u^W(z,t)=0$ for all $t<d(y,z)=r_y(z)$. Thus $r_y\in \mathcal U_b(y)$.

  Next, we show that $r_y$ is maximal. Let $s\in\mathcal U_b(y)$ with $r_y \leqslant s$. By the
  definition of $\mathcal U_b(y)$, we have $s(\cdot)\leqslant N$ and
  $s(z)> 1$ if $z\not\sim y$. Furthermore, there is a nonzero initial
  value $W\in\mathcal W_b(s)$ satisfying
  \begin{equation} \label{uWBndryS}
    u^W(z,t) = 0, \qquad \forall z\in\partial G,\, s(z)\geqslant 2,\,
    t<s(z).
  \end{equation}
  If $s(z)=1$ occurs, it follows from the definition of 
  $\mathcal U_b(y)$ that $z\sim y$,
  i.e. $r_y(z)=1$. Since $r_y\leqslant s$, the wave
  $u^W$ satisfies the following possibly less strict
  set of conditions
  \[
  u^W(z,t) = 0, \qquad \forall z\in\partial G,\,r_y(z)\geqslant 2,\,
  t<r_y(z).
  \]
  This exactly means $W\in\mathcal W_b(r_y)$. Then the same argument above
  yields $\supp(W)\cap G = \{y\}$. Assume that
  $r_y(z') < s(z')$ for some $z'\in\partial G$. This indicates that
  $s(z')\geqslant 2$ as $r_y>0$. Hence
  \eqref{uWBndryS} implies that $u^W(z',t)=0$ for $t < s(z')$, and in
  particular $W(z') = u^W(z',0) = 0$. Then \cref{wavefront} shows that
  $u^W(z',t')\neq 0$ for $t'=d(y,z')<s(z')$, which is a contradiction. Hence $r_y(z')=s(z')$ for all $z'\in\partial G$, and therefore $r_y$ is maximal.
\end{proof}

To uniquely determine $G \cap N(\partial G)$,
we need to find all maximal elements of $\mathcal{U}_b(y)$ for every
$y\in G\cap N(\partial G)$. Then this set of maximal elements contains the set of
boundary distance functions $\{r_y\}_{y\in G\cap N(\partial G)}$,
which corresponds to the initial values
supported only at one single point of $G\cap N(\partial G)$ in the
interior. However in general, as with 
\cref{maximal}, there are more maximal elements than just the boundary
distance functions. 

To reconstruct the graph structure, we need to single out the actual boundary distance
functions from the whole set of maximal elements. We will spend the rest of this subsection to address it.

\begin{definition}\label{arrivalTimeDef}
  We define the \emph{arrival time} of a wave with an initial value
  $W$ at a boundary point $z\in \partial G$, to be the earliest time
  $t\geqslant 1$ when $u^W(z,t)\neq0$. Denote the arrival time at $z$ by $t^W_z$.
\end{definition}

\begin{definition}\label{A0}
  Denote by $\mathcal{A}$ the set of all the $L^2(G)$-normalized
  initial values $W$ satisfying the following three conditions:
  \begin{enumerate}[(1)]
  \item \label{aPropInit} $W:G\cup\partial G\to\R$, $\partial_\nu
    W|_{\partial G} = 0$, i.e. $W$ is an initial value;
  \item \label{aPropMax} $W\in \mathcal{W}(s)$ for some $s\in
    \max(\mathcal{U})$, or $W\in \mathcal{W}_b(s)$ for some $s\in
    \max(\mathcal{U}_b(y))$ and some $y\in G\cap N(\partial G)$,
    i.e. $W$ corresponds to a maximal element;
  \item \label{aPropPos} for all $z\in\partial G$, we have
    $u^W(z,t_z^W)>0$, i.e. the first arrival of the wave at any
    boundary point is with a positive sign.
  \end{enumerate}

For $x\in G$, we use $W_x$ to denote a function satisfying $\partial_{\nu}W_x|_{\partial G}=0$,
$\supp(W_x)\cap G=\{x\}$ and $W_x(x)>0$.
  Finally, we define the set $\mathcal A_0$ as the $L^2$-normalized initial values
  supported at one single point, 
  \[
  \mathcal A_0 = \left\{ \frac{W_x}{\lVert W_x\rVert_{L^2(G)}}
  \,\middle|\, x\in G,\, \partial_{\nu}W_x|_{\partial G}=0,\,
  \supp(W_x)\cap G=\{x\},\, W_x(x)>0 \right\}.
  \]
\end{definition}

Let us remark here regarding the motivation of Definition \ref{A0}.
In the next subsection, we will show in details that the Neumann boundary spectral data determine the sets $\mathcal{U}$ and $\mathcal{U}_b(y)$. This would imply that the Neumann boundary spectral data then determine the Fourier coefficients of functions in $\mathcal{A}$. The functions in $\mathcal{A}$ can be supported at multiple interior points, while what we really want are the functions in $\mathcal{A}_0$ supported at one single interior point. Therefore, we need to construct an algorithm (Lemma \ref{procedure}) to single out $\mathcal{A}_0$ from $\mathcal{A}$.

\begin{lemma} \label{A0inA}
  Let $\mathbb G$ be a finite connected weighted graph with boundary satisfying \cref{assumption}.
  Then
  $\mathcal A_0 \subseteq \mathcal A$.
\end{lemma}
\begin{proof}
  Let $W\in \mathcal{A}_0$. Then $W$ is $L^2(G)$-normalized and it satisfies
  property \ref{aPropInit} in \cref{A0}. Furthermore
  $W=W_x$ for some $x\in G$. We claim that $t_z^W = r_x(z)$ for
  all $z\in\partial G$.

  This claim follows directly from the propagation of the wavefront
  (\cref{wavefront}) if $x\in G - N(\partial G)$, which yields that
  $u^W(z,t_z^W)>0$ for all $z\in\partial G$. If $x\in G\cap N(\partial
  G)$ we see that $t_z^W=r_x(z)$ when $r_x(z)\geqslant 2$ by
  \cref{wavefront}. If $r_x(z)=1$, then $z\sim x$ and $u^W(z,t)$ is
  determined by the Neumann boundary condition \eqref{uValuesOnBndry},
  which gives $u^W(z,0)=u^W(z,1)>0$. Hence $t_z^W=1$ in this case
  by \cref{arrivalTimeDef}. In conclusion, $t_z^W = r_x(z)$ and
  $u^W(z,t_z^W)>0$ for all $z\in\partial G$, i.e. the property
  \ref{aPropPos} of \cref{A0}. Moreover, \cref{maximal,maximalb} imply
  that $t_\cdot^W$ is a maximal element, i.e. the property \ref{aPropMax}
  with $s(z) = t_z^W$.
\end{proof}

Observe that $\mathcal{A}_0$ is an orthonormal basis of the linear span
of $\mathcal{A}$ with respect to the $L^2(G)$-inner product.

\begin{lemma}\label{arrivaltime}
  Let $\mathbb G$ be a finite connected weighted graph with boundary satisfying \cref{as1case2} of \cref{assumption}.
  \begin{enumerate}[(1)]
  \item Given any initial value $W$ satisfying
    $\partial_{\nu}W|_{\partial G}=0$ and the property \ref{aPropPos}
    of \cref{A0}, if $x_0$ is an extreme point of $\supp(W)\cap G$,
    then $W(x_0)>0$.
  
  \item Given any nonzero initial value $W$ satisfying
    $\partial_{\nu}W|_{\partial G}=0$ and $W|_G\geqslant 0$, then for any $z\in\partial G$, we have
    \[
    t^{W}_z=\min_{x\in\supp(W)\cap G}t^{W_x}_z\, .
    \]
  \end{enumerate}
  As a consequence, if $W,W_0$ are two nonnegative initial values
  satisfying the Neumann boundary condition and $\supp(W_0)\cap
  G\subseteq \supp(W)\cap G$, then $t^{W_0}_z\geqslant t^W_z$ for any
  $z\in \partial G$.
\end{lemma}
\begin{proof}
  For the first claim, let $z_0\in \partial G$ be a boundary point
  realizing the extreme point condition of $x_0$. If
  $d(x_0,z_0)\geqslant 2$, the arrival time $t^{W}_{z_0}=d(x_0,z_0)$
  due to \cref{wavefront}. If $d(x_0,z_0)=1$, then $t^{W}_{z_0}=1$ due
  to the Neumann boundary condition for $W$. Hence the property
  \ref{aPropPos}, Lemma \ref{wavefront} and the Neumann boundary
  condition yield $W(x_0)> 0$.
  
  Next, consider the second claim. Due to $W|_G\geqslant 0$, the initial
  value $W$ restricted to $G$ can be written as
  \[
  W|_{G}=\sum_{x\in\supp(W)\cap G} {\alpha}_x W_x|_{G}
  \]
  for some positive numbers $\alpha_x$, where $W_x(x)=1$ and $W_x(y)=0$
  if $y\in G-\{x\}$. Since $W$ and each $W_x$ determine their
  boundary values uniquely and linearly from their values on $G$ by
  \eqref{uValuesOnBndry}, the form above extends to the whole graph
  $G\cup\partial G$. By linearity and the uniqueness of the solution of
  \eqref{waveequation}, the wave $u^W$ has the following form at any
  $z\in \partial G$ and $t\in \N$,
  \[
  u^{W}(z,t)=\sum_{x\in\supp(W)\cap G} \alpha_x u^{W_x}(z,t).
  \]
  Since $u^{W_x}(z,t^{W_x}_z)>0$ for any $z\in \partial G$ by
  \cref{wavefront} and all $\alpha_x$ are positive, we know that the
  earliest time $u^W(z,\cdot)$ becomes nonzero is the earliest time
  when any of $u^{W_x}(z,\cdot)$ becomes nonzero. 
  This shows that for any $z\in\partial G$,
  \[
  t^{W}_z=\min_{x\in\supp(W)\cap G}t^{W_x}_z.
  \]

  The last part of the lemma follows from the condition that
  $\supp(W_0)\cap G\subseteq \supp(W)\cap G$, since a minimum over a
  smaller set can only be larger.
\end{proof}

\begin{lemma}\label{samesign}
  Let $\mathbb G$ be a finite connected weighted graph with boundary satisfying \cref{assumption}.
  If an
  initial value $W$ satisfies $W|_G\geqslant 0$ and $|\supp(W)\cap
  G|\geqslant 2$, then $W\notin \mathcal{A}$.
\end{lemma}
\begin{proof}
  Denote $S=\{x\in G \mid W(x)\neq 0\}$ and we have $|S|\geqslant 2$
  by assumption. We will show that the maximality requirement
  (\ref{aPropMax}) of \cref{A0} fails if (\ref{aPropInit}) is
  satisfied.

  \smallskip
  First, let us bring forth a contradiction from the assumption that
  $W\in \mathcal{W}(s)$ for some $s\in\max(\mathcal{U})$. By the
  Two-Points Condition, there exists $x_1\in S$ and $z_1\in \partial
  G$, such that $x_1$ is the unique nearest point in $S$ from
  $z_1$. Since $W\in\mathcal W(s)$ and $s(z_1)\geqslant 2$ as
  $s\in\mathcal U$, we have $W(z_1)=u^W(z_1,0)=0$. Then
  \cref{wavefront} implies that $t^{W}_{z_1}=d(x_1,z_1)$ and
  $2\leqslant s(z_1)\leqslant t^{W}_{z_1}$. We consider the following
  modified function $s':\partial G\to \Z_+$ defined by
  $s'(z_1)=s(z_1)+1$ and equal to $s$ at all other boundary points.

  Now we prove that $s'\in \mathcal{U}$ and consequently $s$ cannot be
  maximal in $\mathcal{U}$. On one hand, we have $s' \geqslant s \geqslant 2$. 
  On the other hand, we have $s' \leqslant N$. This is because if $s(z_1)=N$ then
  $d(x_1,z_1)\geqslant N$, which means that all $y\in S-\{x_1\}$ are at least distance
  $N+1$ from $z_1$. But this is impossible since there are only $N$
  interior points, considering that distances (precisely $d_{re}$ on the reduced graph $\mathbb{G}_{re}$) are realized by
  paths passing through interior points by \cref{avoidBndry}. Hence
  $2\leqslant s' \leqslant N$ and it remains to show that $\mathcal
  W(s')$ is nontrivial. Define another initial value $W_0$ to be
  $W_0(x_1)=0$ and equal to $W$ elsewhere on $G$. By the propagation
  of the wavefront (\cref{wavefront}), we have $u^{W_0}(z_1,t)=0$ for
  $t\leqslant d(x_1,z_1)$. Since $s(z_1)\leqslant
  t^{W}_{z_1}=d(x_1,z_1)$ and $s'(z_1) = s(z_1)+1$, we have
  $u^{W_0}(z_1,t)=0$ for $t<s'(z_1)$. Since $W|_G\geqslant 0$, the
  arrival time of the wave $u^{W_0}$ at any other boundary point is no
  earlier than that of $u^W$ by \cref{arrivaltime}. This shows
  $W_0\in\mathcal{W}(s')$ and it is a nontrivial
  element since $|S|\geqslant 2$. Hence $s'\in\mathcal{U}$ and $s$ cannot be maximal.
  
  \smallskip
  Next we consider $W\in \mathcal{W}_b(s_b)$ for some $s_b\in
  \mathcal{U}_b(y)$ and show that $s_b$ cannot be maximal. Following
  the previous argument, we can find $x_1\in S$ and $z_1\in \partial
  G$, such that $x_1$ is the unique nearest point in $S$ from $z_1$. If
  $s_b(z_1)\geqslant 2$, then the previous argument applies. Otherwise
  if $s_b(z_1)=1$, then $z_1\sim y$. We define $s_b':\partial G\to
  \Z_+$ by $s_b'(z_1)=2$ and equal to $s_b$ at all other boundary
  points. As before, we see that $s_b' \leqslant N$, and 
  $s_b'(z)=1$ implies $z\sim y$. It remains to show that there is a
  nontrivial initial value $W_0\in\mathcal{W}_b(s_b')$. We choose
  $W_0(x_1)=0$ and equal to $W$ elsewhere on $G$. Since $x_1$
  is an extreme point of $S$ with respect to $z_1$ and $W_0(x_1)=0$,
  we have $W_0(z_1)=0$ by the Neumann boundary condition. This implies that
  $u^{W_0}(z_1,0)=u^{W_0}(z_1,1)=W_0(z_1)=0$, and hence
  $u^{W_0}(z_1,t)=0$ for $t<s_b'(z_1)$. Then the same argument as for the earlier
  case shows that $u^{W_0}(z,t)=0$ for
  $t<s_b'(z)$ when $s_b'(z)\geqslant 2$. Thus we find a nontrivial initial value
  $W_0\in \mathcal W_b(s_b')$. Therefore $s_b$ cannot be maximal
  in $\mathcal U_b(y)$.
\end{proof}

Finally, we use the following criteria to distinguish $\mathcal{A}_0$ from
$\mathcal{A}$.

\begin{lemma}\label{procedure}
  Let $\mathbb G$ be a finite connected weighted graph with boundary satisfying \cref{assumption}.
  Then a subset
  $\tilde{\mathcal{A}}\subseteq \mathcal{A}$ satisfies the following
  two properties
  \begin{enumerate}[(1)]
  \item \label{ONbasis} $\tilde{\mathcal{A}}$ is an orthogonal basis
    of the linear span of $\mathcal{A}$ in $L^2(G)$;
  \item \label{negativity} for any $W\in
    \mathcal{A}-\tilde{\mathcal{A}}$, there exists $\tilde{W}\in
    \tilde{\mathcal{A}}$ such that $\langle W,\tilde{W}
    \rangle_{L^2(G)} <0$,
  \end{enumerate}
  if and only if $\tilde{\mathcal{A}}=\mathcal{A}_0$.
\end{lemma}
\begin{remark*}
  Elements of $\mathcal A$ are normalized, so we are actually searching for an orthonormal basis satisfying Property (\ref{negativity}). Property (\ref{ONbasis}) can also be formulated as follows: $\tilde{\mathcal{A}}$ has cardinality equal to $|G|$, and its elements are mutually orthogonal with respect to the $L^2(G)$-inner product.
\end{remark*}
\begin{proof}
  First, we show that $\mathcal{A}_0$ satisfies these two properties. The set $\mathcal{A}_0$ satisfies Property~(\ref{ONbasis}) as a direct consequence of \cref{A0inA}. Since every function in
  $\mathcal{A}-\mathcal{A}_0$ is supported at multiple interior points by the definition of $\mathcal A_0$, \cref{samesign} implies
  that any function $W\in \mathcal A -\mathcal A_0$ must have a negative
  value at some interior point, say at $x\in G$. Then the condition $\langle W,\tilde{W}
    \rangle_{L^2(G)} <0$ is satisfied with $\tilde W = W_x$. Hence Property \ref{negativity} is satisfied for $\mathcal{A}_0$.

  Next, we prove the ``only if'' direction. \emph{We claim that if $\tilde{\mathcal{A}}\nsubseteq \mathcal{A}_0$,
  then the Properties (\ref{ONbasis}) and (\ref{negativity}) cannot be
  satisfied at the same time.} Suppose $\tilde{\mathcal{A}}\nsubseteq
  \mathcal{A}_0$ and Property~(\ref{ONbasis}) is true. The set
  $\tilde{\mathcal{A}}$ consists of two types of initial values: a)
  initial values supported at one single point in the interior (corresponding to
  the boundary distance functions), and b) initial values supported at
  multiple points in the interior (where interactions occur). Note
  that $\tilde{\mathcal{A}}$ may not
  contain the former type of initial values, but it must contain the
  latter type of initial values since $\tilde{\mathcal A} - \mathcal
  A_0 \neq\emptyset$ by assumption. Property (\ref{ONbasis}) implies that the support of these two types of initial values does not intersect.
  
  Consider the union of the support (intersected with $G$) of all the
  initial values of type b) in $\tilde{\mathcal{A}}$, denoted by
  \[
  S = \bigcup_{{W\in\tilde{\mathcal A},\ W\textrm{ of type
        b)}}} \supp(W) \cap G.
  \]
  By the Two-Points Condition, we can find an extreme point
  $\tilde{x}\in G$ of $S$. Then we consider the $L^2(G)$-normalized initial value
  $W_{\tilde{x}} \in \mathcal{A}$ supported at $\tilde{x}\in G$ with $W_{\tilde x}(\tilde
  x)>0$. Orthogonality implies that
  $W_{\tilde{x}}\notin\tilde{\mathcal{A}}$, or equivalently $W_{\tilde{x}}\in \mathcal{A}-\tilde{\mathcal{A}}$. For any $\tilde{W}\in \tilde{\mathcal{A}}$ with
  $\langle W_{\tilde{x}},\tilde{W} \rangle_{L^2(G)} \neq 0$, we know
  that $\tilde{W}$ is supported at multiple points containing
  $\tilde{x}$. The condition that $\tilde{x}$ is an extreme point of
  $S$ implies that $\tilde{x}$ is also an extreme point of its subset
  $\supp(\tilde{W})\cap G$. Then $\tilde{W}(\tilde{x})>0$ by
  \cref{arrivaltime}(1). As a result, the positivity implies
  \[
  \langle W_{\tilde{x}}, \tilde{W} \rangle_{L^2(G)} =
  \mu_{\tilde{x}}W_{\tilde{x}}(\tilde{x})\tilde{W}(\tilde{x})>0.
  \]
  Hence $\langle W_{\tilde{x}}, \tilde{W} \rangle_{L^2(G)}\geqslant 0$
  for all $\tilde{W}\in \tilde{\mathcal{A}}$. This contradicts
  Property~(\ref{negativity}), and therefore proves the claim.
  
  The claim shows that $\tilde{\mathcal A} \subseteq
  \mathcal A_0$ for any subset $\tilde{\mathcal A} \subseteq \mathcal A$
  satisfying both properties. The set $\mathcal{A}_0$ is an orthogonal basis of the linear
  span of $\mathcal{A}$, and the
  only subset of $\mathcal{A}_0$ also forming a basis
  is $\mathcal{A}_0$ itself. Hence Property (\ref{ONbasis}) yields $\tilde{\mathcal A} = \mathcal A_0$.
\end{proof}



\subsection{Determination from spectral data}

In this subsection, we will tie in the previous subsection's objects to the
spectral and boundary data of a graph. We will show that if two graphs
have the same \emph{a priori} data, then the spectral characterization of
various objects, such as $\mathcal U$, $\mathcal A_0$ from the previous
subsection, of these two graphs coincide. This leads to the conclusion
that the inverse spectral problem is solvable.

Without loss of generality, \emph{we still assume $\mathbb G = \mathbb
G_{re}$ throughout this section}.



\begin{lemma} \label{Wseries}
  Let $\mathbb G$ be a finite connected weighted graph with boundary, and $q$ be a real-valued potential
  function on $G$.  Let $\{\phi_j\}_{j=1}^N$ be the orthonormalized Neumann
  eigenfunctions of $(\mathbb{G},q)$. For any function
  $W: G\cup\partial G \to \R$, denote
  \[
  \widehat{W}(j) = \langle W, \phi_j\rangle_{L^2(G)}, \quad \widehat{W}=\big( \widehat{W}(1),\cdots, \widehat{W}(N) \big) \in \R^N.
  \]
  If $W$ is an initial value for
  \eqref{waveequation}, i.e. satisfying $\partial_\nu W|_{\partial G} = 0$,
  we have
  \[
  W(x) = \sum_{j=1}^N \widehat W(j) \phi_j(x), \quad \forall \,x\in G \cup \partial G.
  \]
  Conversely, given any $(c_j)_{j=1}^N \in \R^N$, $\sum_j c_j \phi_j$ gives an initial value for \eqref{waveequation}.
\end{lemma}
\begin{proof}
Since $\mathbb{G}$ is connected and there are no edges between boundary points by assumption, every boundary point is connected to the interior.
  Then the claims are a direct consequence of the orthonormality of the
  eigenfunctions in $L^2(G)$, \eqref{uValuesOnBndry} and
  $\partial_{\nu}\phi_j|_{\partial G}=0$.
\end{proof}

\begin{notation*}
  Given a complete orthonormal family of Neumann eigenfunctions of $(\mathbb{G},q)$ and
  $V \subseteq  L^2(G)$, we denote
  $  \widehat V = \{ \widehat f\in\R^N \mid f \in V\}.$
  If $\mathbb G'$ is another finite weighted graph with boundary, then we denote by $V'$
  a subset of $L^2(G')$. In this case, $\widehat{V'}$ is defined the same as
  above, but the hat-notation itself is defined using the
  eigenfunctions $\phi_j'$ of $\mathbb G'$ rather than those of
  $\mathbb G$.
\end{notation*}

The following lemma enables us to calculate a wave at any boundary point and any time,
if we know the Neumann boundary spectral data and the Fourier transform (or the spectral representation) of the initial value of the wave.

\begin{lemma} \label{uRepresentation}
  Let $\mathbb G$ be a finite connected weighted graph with boundary, and $q$ be a real-valued potential
  function on $G$.  Let $(\lambda_j,\phi_j)_{j=1}^N$ be the Neumann
  eigenvalues and orthonormalized Neumann eigenfunctions of $(\mathbb{G},q)$.
  Suppose $W$ is the initial value of some wave $u^W$ satisfying the
  wave equation \eqref{waveequation}. Then
  \begin{align*}
    u^W(x,t) &= \sum_{\{j\mid\lambda_j = 0\}} \widehat W(j) \phi_j(x)
    + \sum_{\{j\mid\lambda_j = 4\}} \widehat W(j) (-2t+1)(-1)^t
    \phi_j(x) \\ &\phantom{=} + \sum_{\{j\mid\lambda_j\notin\{0,4\}\}}
    \widehat W(j) \frac{\beta_{j,1}^{t}-\beta_{j,2}^{t} -
      (\beta_{j,1}^{t-1}-\beta_{j,2}^{t-1})}{\beta_{j,1}-\beta_{j,2}} \phi_j(x)\, ,
  \end{align*}
  where
  \beq\label{mu formula}
  \beta_{j,1} = -\frac{\lambda_j}{2} + 1 -
  \sqrt{\left(\frac{\lambda_j}{2} - 1\right)^2 - 1}\; , \quad \beta_{j,2}
  = -\frac{\lambda_j}{2} + 1 + \sqrt{\left(\frac{\lambda_j}{2} -
    1\right)^2 - 1}\; .
  \eeq
  Conversely, given any $\widehat W\in\R^N$, then the wave $u^W$ defined as
  above solves \eqref{waveequation} with the initial value $W =
  \sum_{j=1}^N \widehat W(j) \phi_j$.
\end{lemma}
\begin{proof}
By assumption, every boundary point is connected to the interior.
  The wave satisfies $\partial_\nu u^W|_{\partial G \times \N}=0$, so the orthonormality of $\{\phi_j\}_{j=1}^N$ in
  $L^2(G)$ and \eqref{uValuesOnBndry} imply that
  \[
  u(x,t) = \sum_{j=1}^N a_j(t) \phi_j(x)
  \]
  on $(G\cup\partial G) \times \N$ for some functions $a_j: \N \to
  \R$. The wave equation \eqref{waveequation} and the eigenvalue problem \eqref{neumannEigFunc} yield that
    \beq\label{a formula}
  a_j(t+1) + (\lambda_j-2)a_j(t) + a_j(t-1) = 0
  \eeq
  for all $t\in\Z_+$. The solutions to the associated characteristic equation
  $\beta_j^2 + (\lambda_j-2)\beta_j + 1 = 0$ are shown in the lemma statement.
  The characteristic equation has two identical solutions if $\lambda_j=0$ or $4$, in which case the solutions are $1$ or $-1$. Hence $a_j$ has the following form:
\begin{equation}\label{formula_aj}
  a_j(t) = \begin{cases}
    b_j t + c_j, &\lambda_j = 0,\\
    (b_j t + c_j)(-1)^t, &\lambda_j = 4,\\
    b_j \beta_{j,1}^t + c_j \beta_{j,2}^t, &\lambda_j \notin\{0,4\},
  \end{cases}
  \qquad t\in\N.
\end{equation}

  Recall that $u^W(\cdot,0) = u^W(\cdot,1)$. Then the
  formula for $a_j$ implies that
  \begin{align*}
    &\sum_{\{j\mid\lambda_j = 0\}} c_j \phi_j + \sum_{\{j\mid\lambda_j
      = 4\}} c_j \phi_j + \sum_{\{j\mid\lambda_j\notin\{0,4\}\}} (b_j
    + c_j) \phi_j \\ &= \sum_{\{j\mid\lambda_j = 0\}} (b_j + c_j)
    \phi_j - \sum_{\{j\mid\lambda_j = 4\}} (b_j + c_j) \phi_j +
    \sum_{\{j\mid\lambda_j\notin\{0,4\}\}} (b_j\beta_{j,1} +
    c_j\beta_{j,2}) \phi_j.
  \end{align*}
  Taking the inner product with any $\phi_j$, and the orthonormality of $\phi_j$ allows
  us to solve $b_j$ as a function of $c_j,\beta_{j,1}$ and $\beta_{j,2}$
  for each $j=1,\ldots,N$. This gives
  \[
  b_j = \begin{cases}
    0, & \lambda_j = 0,\\
    -2c_j, & \lambda_j = 4,\\
    -c_j \frac{\beta_{j,2}-1}{\beta_{j,1}-1}, & \lambda_j \notin\{0,4\}.
  \end{cases}
  \]
  Note that $\beta_{j,1},\beta_{j,2}\neq \pm 1$ if $\lambda_j\notin \{0,4\}$.
  Hence we obtain the formula for the wave:
  \begin{align*}
    u^W(x,t) &= \sum_{\{j\mid\lambda_j = 0\}} c_j \phi_j(x) +
    \sum_{\{j\mid\lambda_j = 4\}} c_j (-2t+1)(-1)^t \phi_j(x) +
    \\ &\phantom{=} + \sum_{\{j\mid\lambda_j\notin\{0,4\}\}} c_j
    \left( - \frac{\beta_{j,2}-1}{\beta_{j,1}-1}\beta_{j,1}^t + \beta_{j,2}^t
    \right) \phi_j(x)\, ,
  \end{align*}
  for $x\in G\cup\partial G$ and $t\in\N$. This satisfies the second
  initial condition, the Neumann boundary condition and the wave equation in
  \eqref{waveequation}. By \cref{Wseries}, the first initial condition $u^W(x,0)=W(x)$
  gives that
  \[
  c_j = \begin{cases}
    \widehat W(j), &\lambda_j = 0,\\
    \widehat W(j), &\lambda_j = 4,\\
    \widehat W(j) \frac{\beta_{j,1}-1}{\beta_{j,1}-\beta_{j,2}}, &\lambda_j \notin \{0,4\}.
  \end{cases}
  \]
  The first claim of the lemma follows after plugging these into the formula for the wave. 
  The converse claim is a
  straightforward calculation whose details are actually scattered in
  the proof of the first claim.
\end{proof}

In our setting, we are working with two
graphs having the same boundary and the same Neumann boundary spectral data.
For convenience, we make use of the following pullback notation.

\begin{notation*}
  Given two finite weighted graphs with boundary $\mathbb G,\mathbb G'$ and a
  boundary-isomorphism $\Phi_0$ (\cref{aPrioriGeom}), we define the
  following notation.
  \begin{itemize}
  \item For $f' : \partial G' \to \R$, we denote $\Phi_0^\ast f' = f'
    \circ (\Phi_0|_{\partial G})$.
  \item If $S'$ is a set of functions on $\partial G'$, denote $\Phi_0^\ast S' =
    \{ \Phi_0^\ast f' \mid f' \in S'\}$.
  \end{itemize}
  This notation defines $\Phi_0^\ast f' : \partial G \to \R$, and
  $\Phi_0^\ast S'$ as a set of functions on $\partial G$.
\end{notation*}

We consider initial values not just as
functions on the graph, but also as abstract points in $\R^N$ using
their Fourier series representation in \cref{Wseries}.
\cref{uRepresentation} shows that the spectral (Fourier)
coefficients of an initial value $W$ uniquely determine the boundary
values of the corresponding wave $u^W$. 

\begin{lemma} \label{sameOnBndry}
  Let $\mathbb G, \mathbb G'$ be two finite connected weighted graphs with boundary, 
  and $q,q'$ be real-valued potential functions on $G, G'$. 
  Suppose $(\mathbb{G},q)$ is spectrally
  isomorphic to $(\mathbb{G}',q')$ with a boundary-isomorphism $\Phi_0$, namely
  \[
  (\lambda_j,\phi_j|_{\partial G})_{j=1}^N =
  (\lambda_j',\Phi_0^\ast(\phi_j|_{\partial G'}))_{j=1}^{N}.
  \]
  Let $(c_j)_{j=1}^N \in\R^N$,
  $
  W = \sum_{j=1}^N c_j \phi_j,$ $ W' = \sum_{j=1}^N c_j \phi_j',
  $
  and $u^W,u^{W'}$ be the corresponding solution to the wave equation
  \eqref{waveequation} in $\mathbb G, \mathbb G'$. Then
  $
  u^W(z,t) = u^{W'}(\Phi_0(z),t)
  $
  for all $z\in\partial G$ and $t\in\N$.
\end{lemma}
\begin{proof}
  This is a direct consequence of the representation formula for $u^W,
  u^{W'}$ in \cref{uRepresentation}, since $\widehat W(j) = c_j =
  \widehat {W'}(j)$, $\lambda_j=\lambda_j'$ and $\phi_j =
  \phi_j'\circ\Phi_0$ on $\partial G$.
\end{proof}
We remark that a full
boundary-isomorphism is not needed for this lemma; a simple bijection
$\partial G \to \partial G'$ which makes the boundary spectral data
equivalent is enough.

\smallskip
Our next few tasks are to show that various objects from
\cref{characterizationObjects} are equivalent, or that their spectral
representations are the same for two spectrally isomorphic
graphs. Recall the definitions of the various objects $\mathcal W,
\mathcal W_b, \mathcal U, \mathcal U_b$ which were defined just before
\cref{maximal,maximalb}. The following lemma shows that knowing the Neumann boundary spectral data leads to the knowledge of the sets $\mathcal{U}, \mathcal{U}_b$. 

\begin{lemma} \label{Wsame}
    Let $\mathbb G, \mathbb G'$ be two finite connected weighted graphs with boundary, 
  and $q,q'$ be real-valued potential functions on $G, G'$. 
  Suppose $(\mathbb{G},q)$ is spectrally
  isomorphic to $(\mathbb{G}',q')$ with a boundary-isomorphism $\Phi_0$. Then
  $
  \widehat{\mathcal W}(s) = \widehat{\mathcal W'}(s')$ and
$  
\widehat{\mathcal W_b}(s) = \widehat{\mathcal W_b'}(s')
$
  for all $s:\partial G \to \Z_+$ and $s' = s\circ(\Phi_0|_{\partial
    G})^{-1}$. As a consequence,
  \[
  \mathcal U = \Phi_0^\ast \mathcal U', \qquad \mathcal U_b(y) =
  \Phi_0^\ast \big( \mathcal{U}_b'( y') \big),
  \]
  for all $y\in G\cap N(\partial G)$ and $y' = \Phi_0(y)$.
\end{lemma}
\begin{proof}
  We use the notation $s' = s\circ(\Phi_0|_{\partial G})^{-1}$
  throughout this proof. By symmetry, it suffices to prove that
  $\widehat{\mathcal W}(s) \subseteq \widehat{\mathcal W'}(s')$. Suppose
  $\widehat W=(c_j)_{j=1}^N \in \R^N$ for some $W\in\mathcal W(s)$. This gives
  $u^W(z,t)=0$ for all $z\in\partial G$ and $t<s(z)$. Then take the function $W' =
  \sum_{j=1}^N c_j \phi_j'$ with the same Fourier coefficients. 
  By \cref{sameOnBndry}, we see that
  \[
  u^{W'}(\Phi_0(z),t) = u^W(z,t) = 0
  \]
  for $z\in\partial G$ and $t<s(z) = s'(\Phi_0(z))$. As $z$ runs
  through $\partial G$, the point $\Phi_0(z)$ runs through the whole
  $\partial G'$. Hence $W' \in \mathcal W'(s')$ and $(c_j)_{j=1}^N = \widehat{W'}
  \in \widehat{\mathcal W'}(s')$. The same argument shows that
  $\widehat{\mathcal W_b}(s) \subseteq \widehat{\mathcal W_b'}(s')$.

  For the claim on $\mathcal U,\mathcal U'$, notice that
  \[
  \dim(\mathcal W(s)) = \dim(\widehat{\mathcal W}(s)) =
  \dim(\widehat{\mathcal W'}(s')) = \dim(\mathcal W'(s')).
  \]
  This is because the map $W \mapsto \widehat W$ from
  $L^2(G)$ to $\R^N$ is an invertible linear map by \cref{Wseries}. 
  Thus $s\in\mathcal
  U$ if and only if $s'\in \mathcal U'$.

  If $s\in\mathcal U_b(y)$ for some $y\in G\cap N(\partial G)$, then
  $\dim(\mathcal W_b'(s')) = \dim(\mathcal
  W_b(s)) \neq 0$. 
  By definition, $s(z)=1$ only if $z\sim y$. By the definition of $\Phi_0$,
  $\Phi_0(z)\sim' \Phi_0(y)$ holds if and only if $z\sim y$. Hence $s'$
  satisfies all conditions of $\mathcal U_b'(y')$.
\end{proof}

The set $\mathcal A$ contains the initial values for which the
corresponding wavefront reaches the boundary with positive
values everywhere and as late as possible (Recall \cref{A0}). The following lemma
shows that knowing the boundary spectral data leads to the knowledge
of the spectral data of all such initial values.
\begin{lemma} \label{Asame}
   Let $\mathbb G, \mathbb G'$ be two finite connected weighted graphs with boundary, 
  and $q,q'$ be real-valued potential functions on $G, G'$. 
  Suppose $(\mathbb{G},q)$ is spectrally
  isomorphic to $(\mathbb{G}',q')$. Then
  $\widehat{\mathcal A} = \widehat{\mathcal A'}$.
\end{lemma}
\begin{proof}
  Let $\Phi_0$ be a boundary isomorphism making the graphs spectrally
  isomorphic. Let $W\in\mathcal A$, i.e. $\lVert W \rVert_{L^2(G)}=1$
  and it satisfies the three conditions in \cref{A0}. Let $W' = \sum_j
  \widehat{W}(j) \phi_j'$. Then $\partial_\nu W'|_{\partial G'}=0$
  and
  \[
  \lVert W' \rVert_{L^2(G')}^2 = \sum_{j=1}^N |\widehat W(j)|^2 =
  \lVert W \rVert_{L^2(G)}^2 = 1.
  \]
  Since $u^W(z,t) = u^{W'}(\Phi_0(z),t)$ for all $z\in\partial
  G$ and $t\in\N$ by \cref{sameOnBndry}, we have
  $u^{W'}(z',t_z^{W'}) > 0$ for all $z'\in\partial G'$. It remains to
  verify Condition~(\ref{aPropMax}) for $W'$, i.e. that $W'$
  corresponds to a maximal element $s':\partial G'\to\N$.

  By \cref{Wsame}, we have for any $y\in G\cap N(\partial G)$,
  $$\max(\mathcal U) = \max(\Phi_0^\ast\mathcal
  U'),\quad \max(\mathcal U_b(y)) = \max \big(\Phi_0^\ast\mathcal
  U_b'(\Phi_0(y))\big).$$
  If
  $W\in\mathcal W(s)$ for some $s\in\max(\mathcal U)$, then
  \[
  \widehat{W'} = \widehat W \in \widehat{\mathcal W}(s) =
  \widehat{\mathcal W'}(s')
  \]
  for $s' = s \circ (\Phi_0|_{\partial G})^{-1}$ by
  \cref{Wsame}. Hence $W'\in\mathcal W'(s')$ for the given $s$ which
  is a maximal element of $\mathcal U$. Next we show
  that $s'$ is a maximal element of $\mathcal U'$. Since
  $\Phi_0^\ast s' = s$, $s \in \max(\mathcal U)$ and 
  $\mathcal U = \Phi_0^\ast \mathcal U'$ by \cref{Wsame}, we see that $\Phi_0^\ast s'
  \in \max(\Phi_0^\ast \mathcal U')$. The pullback does not affect the partial
  order, and hence $s' \in
  \max(\mathcal U')$.

  Similarly, if $W\in\mathcal W_b(s)$ for some $s\in\max(\mathcal
  U_b(y))$ and $y\in G\cap N(\partial G)$, then we see that
  $W'\in\mathcal W_b'(s')$ for $\Phi_0^\ast s' = s$. Then $\Phi_0^\ast s'\in\max\big(\Phi_0^\ast\mathcal
  U_b'(\Phi_0(y))\big)$, which implies that
  $s'\in\max(\mathcal U_b'(\Phi_0(y)))$. Moreover, by the definition of
  $\Phi_0$ (\cref{aPrioriGeom}), we have $\Phi_0(y) \in G' \cap
  N(\partial G')$. Hence $W'\in\mathcal W_b'(s')$ with
  $s'\in\max(\mathcal U_b'(y'))$ for some $y'\in G'\cap N(\partial
  G')$, just as required in Condition~(\ref{aPropMax}) for $\mathcal
  A'$. Thus $W\in\mathcal A$ implies that $W'\in\mathcal A'$, where
  $\widehat{W'} = \widehat{W}$. A symmetric proof shows
  $\widehat{\mathcal A'} \subseteq \widehat{\mathcal A}$.
\end{proof}

With \cref{Asame}, we can finally apply \cref{procedure} to deduce the set of
Fourier coefficients corresponding to initial values supported at one
single interior point. These initial values correspond to individual points of the graph. 

\begin{proposition} \label{A0same}
Let $\mathbb G, \mathbb G'$ be two finite connected weighted graphs with boundary
  satisfying \cref{assumption}, 
  and $q,q'$ be real-valued potential functions on $G, G'$. 
  Suppose $(\mathbb{G},q)$ is spectrally
  isomorphic to $(\mathbb{G}',q')$. Then
  $\widehat{\mathcal A_0} = \widehat{\mathcal A_0'}$.
\end{proposition}
\begin{proof}
  We need to define a subset of $L^2(G')$, such that its Fourier transform is
  equal to $\widehat{\mathcal A_0}$ and satisfies the conditions
  in \cref{procedure}. To write relevant notations clearly, we use $\mathscr F$ to denote the Fourier transform in
  this proof. Let
  \[
  \widetilde{\mathcal A'} = \bigg\{\sum_{j=1}^N c_j \phi_j'
  \,\bigg| \, (c_j)_{j=1}^N \in \mathscr F{\mathcal A_0} \bigg\}.
  \]

  For $f,g\in L^2(G)$, we know that $\langle f,g \rangle_{L^2(G)} = \sum_j
  \mathscr F f(j) \mathscr F g(j) = \mathscr F f \cdot \mathscr F g$,
  where $\cdot$ is the inner product in $\R^N$. Furthermore,
  \[
  \mathscr F {\operatorname{span} \mathcal A} = \operatorname{span}
  \mathscr F{\mathcal A}.
  \]
  Since $\mathcal A_0$ is an orthonormal basis of
  $\operatorname{span} \mathcal A$, these two observations indicate that
  $\mathscr F{\mathcal A_0}$ is an orthonormal basis of
  $\operatorname{span}\mathscr F{\mathcal A}$. The latter is equal to
  $\operatorname{span} \mathscr F{\mathcal A'}$ by \cref{Asame}. Hence
  we can deduce that $\mathscr F{\widetilde{\mathcal A'}} = \mathscr
  F{\mathcal A_0}$ is an orthonormal basis of
  $\operatorname{span}\mathscr F{\mathcal A'}$. Thus 
  $\widetilde{\mathcal A'}$ is an
  orthonormal basis of $\operatorname{span}\mathcal
  A'$. Condition~(\ref{ONbasis}) has been verified.

  Next, let us verify Condition~(\ref{negativity}) in \cref{procedure}.
  Let $W' \in \mathcal A' - \widetilde{\mathcal A'}$ and $W
  = \sum_j \mathscr F{W'}(j) \phi_j$. Since $\mathscr
  F{\mathcal A} = \mathscr F{\mathcal A'}$ by \cref{Asame}, we have
  \[
  \mathscr F W = \mathscr F {W'} \in \mathscr F{\mathcal A'} -
  \mathscr F{\widetilde{\mathcal{A}'}} = \mathscr F{\mathcal A} - \mathscr
  F{\mathcal{A}_0},
  \]
  which shows $W \in \mathcal A - \mathcal A_0$. By \cref{procedure}, there exists
  $\widetilde{W} \in \mathcal A_0$ such that $\langle W, \widetilde W
  \rangle_{L^2(G)} < 0$. Take $\widetilde{W'} = \sum_j
  \mathscr F{\widetilde{W}}(j) \phi_j'$. Then $\widetilde{W'} \in
  \widetilde{\mathcal A'}$ by the latter's definition. Moreover,
  \[
  \langle W', \widetilde{W'} \rangle_{L^2(G)} = \mathscr F{W'} \cdot
  \mathscr F{\widetilde{W'}} = \mathscr F{W} \cdot \mathscr
  F{\widetilde{W}} = \langle W,\widetilde W\rangle_{L^2(G)} < 0,
  \]
  which shows that Condition~(\ref{negativity}) in \cref{procedure} holds for
  $\widetilde{\mathcal A'}$. Hence \cref{procedure} yields $\widetilde{\mathcal A'}
  = \mathcal A_0'$, and the lemma follows.
\end{proof}

Recall that $\mathcal A_0$ is the set of normalized initial values
supported at one single interior point. Since the Fourier transforms of
these sets are the same, it makes sense to identify interior vertices via their Fourier transforms. 
We will show that this identification gives the desired bijection $\Phi$ in Theorem \ref{structureThm}.

\begin{lemma} \label{oneToOne}
  Let $\mathbb G, \mathbb G'$ be two finite connected weighted graphs with boundary
  satisfying \cref{assumption}, 
  and $q,q'$ be real-valued potential functions on $G, G'$. 
  Suppose $(\mathbb{G},q)$ is spectrally
  isomorphic to $(\mathbb{G}',q')$. 
  We define a relation $\equiv$ on $G \times G'$ by
  \[
  x\equiv x' \quad\Longleftrightarrow\quad \widehat{W_x} =
  \widehat{W_{x'}},\;\, \widehat{W_x}\in\widehat{\mathcal A_0}, \;\, \widehat{W_{x'}}\in\widehat{\mathcal A_0'}\, .
  \]
  Then $\equiv$ is a
  one-to-one correspondence.
\end{lemma}
\begin{proof}
  Let us verify that $\equiv$ satisfies the conditions for a one-to-one
  correspondence. We make use of \cref{A0same} which gives 
  $\widehat{\mathcal A_0} = \widehat{\mathcal A_0'}$.

  ``Every $x\in G$ is paired with exactly one $x'\in G'$'': Let $x\in
  G$. Then $\widehat{W_x}\in\widehat{\mathcal A_0}=\widehat{\mathcal
    A_0'}$. The latter set consists of all elements of the form
  $\widehat{W_{x'}}$ for $x'\in G'$. Since the Fourier transform is
  invertible, there is a unique $x'\in G'$ such that
  $\widehat{W_{x'}}=\widehat{W_x}$.

  ``For any $x'\in G'$, there exists a unique $x\in G$ such that $x\equiv
  x'$'': the same proof as above.
\end{proof}

\begin{definition} \label{phiDef}
  With the assumptions of \cref{oneToOne}, given a boundary-isomorphism $\Phi_0$ which makes $(\mathbb G,q)$ and $(\mathbb G',q')$
  spectrally isomorphic, we define a bijective map
  $\Phi:G\cup\partial G \to G'\cup\partial G'$ by
  \[
  \Phi(x) = \begin{cases}
    \Phi_0(x), &x\in\partial G,\\
    x', &x\in G,\, x\equiv x'.
  \end{cases}
  \]
\end{definition}


With this map $\Phi$, we have the point-equivalence between these two graphs.
Now we show
that $\Phi$ also preserves the edge structure.
From what we have done in this section,
the boundary spectral data
provide the knowledge of the Fourier transform of initial values, boundary
values of waves, and the inner product of waves. We use this information to
determine if there is an edge between two points.
\begin{lemma} \label{boundaryConnection}
  Let $\mathbb G$ be a finite weighted graph with boundary, and $x\in G$,
  $z\in\partial G$. Then $x\sim z$ if and only if $W_x(z) \neq 0$.
\end{lemma}
\begin{proof}
  This directly follows from $\partial_\nu W_x |_{\partial G}= 0$ and \eqref{uValuesOnBndry}.
\end{proof}

\begin{lemma} \label{insideConnection}
  Let $\mathbb G$ be a finite weighted graph with boundary satisfying
  \cref{as1case2} of \cref{assumption}, and $x,y\in G$. Then $x\sim y$ if and only if
  \[
  \min\big\{t\in\N \,\big|\, \langle u^{W_x}(\cdot,t), W_y \rangle_{L^2(G)}
  \neq 0 \big\} = 2.
  \]
\end{lemma}
\begin{proof}
Recall that $W_x$ stands for a function on $G\cup \partial G$ satisfying $\partial_{\nu}W_x|_{\partial G}=0$, $\textrm{supp}(W_x)\cap G=\{x\}$ and $W_x(x)>0$.
  By \cref{as1case2} of \cref{assumption} and calculating the wave $u^{W_x}$ from
  \eqref{waveequation} up to time $t=2$, we see that
  \[
  G \cap \supp\big( u^{W_x}(\cdot,t) \big) \subseteq \begin{cases}
    \{x\}, &t\leqslant 1,\\ G \cap N(x,1), &t=2.
  \end{cases}
  \]
Note that \cref{as1case2} of \cref{assumption} is essential to the claim of the support at $t=2$. 
More precisely, in the interior, the wave $u^{W_x}(\cdot,2)$ can be nonzero only at interior points adjacent to $x$, and interior points adjacent to boundary points that are adjacent to $x$. Under \cref{as1case2} of \cref{assumption}, the interior points in the latter case are also adjacent to $x$.
  
  Moreover, we have $u^{W_x}(\cdot,t)>0$ on $\big(G\cap
  N(x,1) \big) - \{x\}$ at $t=2$.
Indeed, for any $p\in \big(G\cap
  N(x,1) \big) - \{x\}$, observe that $\big(\Delta_G u^{W_x}\big)(p,1)>0$. This is because $u^{W_x}(p,1)=W_x(p)=0$, and the only nonzero contributions to the Laplacian come from point $x$ and boundary points adjacent to $x$. One can see that both types of contributions to the Laplacian are positive due to $W_x(x)>0$ and the boundary value given by \eqref{uValuesOnBndry}. 
Moreover, the positive contribution from point $x$ is always present since $p\in N(x,1)-\{x\}$, i.e. $p\sim x$.  
Thus, the Laplacian satisfies $\big(\Delta_G u^{W_x}\big)(p,1)>0$, which yields that the wave $u^{W_x}(p,2)>0$ at $t=2$.

  If $x\sim y$, then $y\in \big(G\cap
  N(x,1) \big) - \{x\}$ and hence the minimum in question
   is equal to $2$. If $x\not\sim y$, then either $x=y$ in which
  case the minimum is $0$, or $d(x,y)\geqslant 2$ in which case $y\notin G \cap
  \supp(u^{W_x}(\cdot,t))$ for $t\leqslant 2$, and the minimum is more
  than $2$.
\end{proof}

\medskip
Finally, we are ready to prove the main theorems.

\begin{proof}[Proof of \cref{structureThm}]
  The first property of $\Phi|_{\partial G}$ being identical to
  $\Phi_0|_{\partial G}$ follows by definition. It remains to verify
  the second property that the edge relations are
  preserved by $\Phi$ and its inverse. Let $p_1,p_2\in G\cup\partial
  G$.

  If $p_1,p_2\in G$, then by \cref{insideConnection},
  \[
  p_1\sim p_2 \quad\Longleftrightarrow\quad \min \big\{t\in\N \,\big|\, \langle
  u^{W_{p_1}}(\cdot,t), W_{p_2} \rangle_{L^2(G)} \neq 0 \big\} = 2.
  \]
  Let us write the inner product using
  the Fourier transform of the initial values $W_{p_1},W_{p_2}$. By \cref{uRepresentation},
  \begin{align*}
    u^{W_{p_1}}(\cdot,t) &= \sum_{\{j\mid\lambda_j = 0\}}
    \widehat{W_{p_1}}(j) \phi_j + \sum_{\{j\mid\lambda_j = 4\}}
    \widehat{W_{p_1}}(j) (-2t+1)(-1)^t \phi_j \\ &\phantom{=} +
    \sum_{\{j\mid\lambda_j\notin\{0,4\}\}} \widehat{W_{p_1}}(j)
    \frac{\beta_{j,1}^{t}-\beta_{j,2}^{t} -
      (\beta_{j,1}^{t-1}-\beta_{j,2}^{t-1})}{\beta_{j,1}-\beta_{j,2}} \phi_j \, .
  \end{align*}
  Taking the inner product with $W_{p_2} = \sum_j
  \widehat{W_{p_2}} \phi_j$ yields
  \begin{align*}
    &\langle u^{W_{p_1}}(\cdot,t), W_{p_2} \rangle_{L^2(G)}
    \\ &\qquad= \sum_{\{j\mid\lambda_j = 0\}} \widehat{W_{p_1}}(j)
    \widehat{W_{p_2}}(j) + \sum_{\{j\mid\lambda_j = 4\}}
    \widehat{W_{p_1}}(j) \widehat{W_{p_2}}(j) (-2t+1)(-1)^t
    \\ &\qquad\phantom{=} + \sum_{\{j\mid\lambda_j\notin\{0,4\}\}}
    \widehat{W_{p_1}}(j) \widehat{W_{p_2}}(j)
    \frac{\beta_{j,1}^{t}-\beta_{j,2}^{t} -
      (\beta_{j,1}^{t-1}-\beta_{j,2}^{t-1})}{\beta_{j,1}-\beta_{j,2}} \\ &\qquad=
    \langle u^{W_{\Phi(p_1)}}(\cdot,t), W_{\Phi(p_2)}
    \rangle_{L^2(G')}.
  \end{align*}
  The second equality is because $\widehat{W_{p_1}}=\widehat{W_{\Phi(p_1)}}$,
  $\widehat{W_{p_2}} = \widehat{W_{\Phi(p_2)}}$, considering that $p_1\equiv
  \Phi(p_1)$, $p_2\equiv \Phi(p_2)$ and the eigenvalues are the same for
  $\mathbb G,\mathbb G'$. Hence the minimal time in question is
  equal to $2$ for $p_1,p_2$ in $\mathbb G$ if and only if it is
  so for $\Phi(p_1),\Phi(p_2)$ in $\mathbb G'$.

  Next, consider the edge between an interior point and a boundary point.
  If $p_1\in G$, $p_2\in\partial G$, then by \cref{boundaryConnection},
  \[
  p_1 \sim p_2 \quad\Longleftrightarrow\quad W_{p_1}(p_2) \neq 0.
  \]
  Since $p_1 \equiv \Phi(p_1)$, $p_2 \equiv \Phi(p_2)$
  and $\phi_j = \phi_j' \circ \Phi_0 = \phi_j' \circ \Phi$ on
  $\partial G$, we have
  \[
  W_{p_1}(p_2) = \sum_{j=1}^N \widehat{W_{p_1}}(j) \phi_j(p_2) =
  \sum_{j=1}^N \widehat{W_{\Phi(p_1)}}(j) \phi_j'\big(\Phi(p_2)\big) =
  W_{\Phi(p_1)}\big(\Phi(p_2)\big).
  \]
  Hence $p_1\sim p_2$ if and only if $\Phi(p_1) \sim'
  \Phi(p_2)$. 
  
  Finally, the
  case of $p_1,p_2\in\partial G$ is trivial, because 
  $\Phi(p_1),\Phi(p_2)\in\partial G'$ and 
  the edge structure on the boundary is \emph{a priori} given.
\end{proof}

\medskip
In \cref{coefficientThm}, we assume that the isomorphic structure
is already known and the vertices of $\mathbb G'$ have been identified with
vertices of $\mathbb G$ via the $\Phi$-correspondence. 
In terms of notations, a vertex $x$ of $\mathbb G$ can also denote a vertex of $\mathbb G'$,
which exactly refers to the vertex $\Phi(x)$ of $\mathbb G'$.

\begin{proof}[Proof of \cref{coefficientThm}]
  Recall from Definition \ref{A0} that
  $W_x\in \mathcal{A}_0$, $W'_x\in \mathcal{A}_0'$ are defined as the $L^2$-normalized initial values satisfying
  $\partial_\nu W_x |_{\partial G} = \partial_\nu W'_x |_{\partial G'}=0$ with
  $G\cap\supp(W_x) = G'\cap\supp(W'_x) = \{x\}$. Then the definition of the $L^2$-norm \eqref{innerproduct} yields
  $
  W_x(x) = \mu_x^{-1/2}$ and $W'_x(x) = {\mu'}_x^{-1/2}.
 $

  Now let us prove (\ref{coefCase1}). Assume $\mu=\mu'$. First, we prove $g=g'$.
  
  For $x,y\in G$ with $x\sim y$, we have
  \begin{align}
    &\big\langle (-\Delta_G +q)W_x, W_y\big\rangle_{L^2(G)} = \mu_y
    W_y(y) \big((-\Delta_G+q)W_x\big)(y) = \mu_y W_y(y)
    (-\Delta_G W_x)(y) \notag \\ &\qquad = - W_y(y)
    \sum_{\substack{p\sim y\\p\in G\cup\partial G}} g_{yp} \big(
    W_x(p) - W_x(y) \big) = - W_y(y) \sum_{\substack{p\sim y\\p\in
        G\cup\partial G}} g_{yp} W_x(p). \label{gDetermination}
  \end{align}
  On the other hand, this inner product can be determined from the
  spectral data. Namely by (\ref{neumannEigFunc}), we have
  \beq
 &   &\big\langle (-\Delta_G +q)W_x, W_y\big\rangle_{L^2(G)} =
    \bigg\langle \sum_j \widehat{W_x}(j) \lambda_j \phi_j, \sum_j
    \widehat{W_y}(j) \phi_j \bigg\rangle_{L^2(G)} \notag \\ & & =
    \sum_j \lambda_j \widehat{W_x}(j) \widehat{W_y}(j) = \sum_j
    \lambda'_j \widehat{W'_x}(j) \widehat{W'_y}(j) =
    \big\langle (-\Delta_{G'} +q')W'_x,
    W'_y\big\rangle_{L^2(G')} \, . \label{gDetermination2}
  \eeq
  since $\lambda_j=\lambda'_j$, $\widehat{W_x} = \widehat{W'_x}$ and
  $\widehat{W_y} = \widehat{W'_y}$, due to \cref{oneToOne,phiDef}.

\smallskip
We consider two cases from here. 
  
  {\bf (i)} If $x\in
  G-N(\partial G)$ in \eqref{gDetermination}, then for all $p\in G\cup\partial G$, 
  we have $W_x(p) = 0$
  unless $p=x$ (since $x$ is not adjacent to the boundary). 
  Then \eqref{gDetermination} and \eqref{gDetermination2} yield
  \begin{align*}
    -W_y(y) W_x(x) g_{xy} &= \big\langle (-\Delta_G +q)W_x,
    W_y\big\rangle_{L^2(G)} = \big\langle (-\Delta_{G'} +q')W'_x,
    W'_y\big\rangle_{L^2(G')} \\ &= -W'_y(y) W'_x(x) g'_{xy}.
  \end{align*}
  This implies that $g_{xy} = g'_{xy}$, since
  $W_y(y) = \mu_y^{-1/2} =
  {\mu'}_y^{-1/2} = W'_y(y)$ and similarly $W_x(x)=W'_x(x)$.
  
  {\bf (ii)} It remains to consider the case where $x,y\in
  G\cap N(\partial G)$. In this case, \eqref{gDetermination} and
  \eqref{gDetermination2} reduce to
  \beq\label{gDetermination3}\\
     W_y(y)W_x(x) g_{xy} +W_y(y)\hspace{-3mm} \sum_{\substack{z\sim y, z\sim
        x\\ z\in\partial G}} g_{yz} W_x(z)  = 
    W'_y(y)W'_x(x) g'_{xy} +W'_y(y)\hspace{-3mm} \sum_{\substack{z\sim y, z\sim
        x\\ z\in\partial G}} g'_{yz} W'_x(z). \nonumber\hspace{-1cm}
  \eeq
  Since $\partial_\nu W_x(z) = \partial_\nu W'_x(z) = 0$, by
  \eqref{uValuesOnBndry} we see that
  \[
  W_x(z) = \frac{g_{xz}W_x(x)}{\sum_{{p\sim z,p\in G}}
    g_{pz}} = \frac{g'_{xz}W'_x(x)}{\sum_{{p\sim z,p\in G}}
    g'_{pz}} = W'_x(z),
  \]
  where we have used that $W_x(x)=W'_x(x)$ and $g=g'$ on edges from the boundary.
 By \eqref{gDetermination3}, we see that $g_{xy}=g'_{xy}$. This concludes the unique determination of $g$.

  Next, we prove $q=q'$, assuming
  $\mu=\mu'$. Let $x\in G$. The spectral data
  determines the following inner product:
  \begin{align*}
    &\langle (-\Delta_G+q)W_x,W_x \rangle_{L^2(G)}=\mu_x
    W_x(x)\big((-\Delta_G+q)W_x\big)(x)
    \\ &\qquad=-W_x(x)\sum_{\substack{p\sim x\\p\in G\cup\partial G}}
    g_{xp}\big(W_x(p)-W_x(x)\big)+q(x) .
  \end{align*}
  Observe that all relevant quantities have already been uniquely determined as above, during the proof for the unique determination of $g$. Hence $q=q'$.

  \smallskip
  At last, we prove (\ref{coefCase2}). Assume $q=q'=0$. For the graph Laplacian (with zero potential), there exists $j_0$
  such that $\lambda_{j_0}=\lambda'_{j_0}=0$ and
  $\phi_{j_0}=\phi'_{j_0} = c$ for some constant $c\in\R$. Given any
  $x\in G$, we have $\langle W_x, \phi_{j_0}\rangle_{L^2(G)} = \mu_x
  W_x(x) c$, where $W_x(x) =
  \mu_x^{-1/2}$. Then
  \[
  \mu_x = \frac{\langle W_x, \phi_{j_0}\rangle_{L^2(G)}^2}{c^2} =
  \frac{\big(\widehat{W_x}(j_0)\big)^2}{c^2} =
  \frac{\big(\widehat{W'_x}(j_0)\big)^2}{c^2} = \frac{\langle W'_x,
    \phi'_{j_0}\rangle_{L^2(G')}^2}{c^2} =\mu'_x.
  \]
  Hence $\mu=\mu'$. Then the
  assumption of (\ref{coefCase1}) is satisfied and therefore $g=g'$.

\smallskip  
  In particular, if $\mu={\rm deg}_{\mathbb{G}}$, $\mu'={\rm deg}_{\mathbb{G}'}$, then $\mu=\mu'$ since $\Phi$ preserves the edge structure by \cref{structureThm}. Hence the conclusion follows from (\ref{coefCase1}).
\end{proof}

\end{document}